\documentclass[12pt]{amsart}
\usepackage{amssymb}
\usepackage{hyperref}

\newtheorem{theorem}{Theorem}[section]
\newtheorem{definition}[theorem]{Definition}
\newtheorem{conjecture}[theorem]{Conjecture}
\newtheorem{proposition}[theorem]{Proposition}

\newtheorem{problem}[theorem]{Problem}
\newtheorem{lemma}[theorem]{Lemma}

\begin{document}

\title{Combinatorial aspects of orthogonal group integrals}

\author{Teodor Banica}
\address{T.B.: Department of Mathematics, Cergy-Pontoise University, 95000 Cergy-Pontoise, France. {\tt teodor.banica@u-cergy.fr}}

\author{Jean-Marc Schlenker}
\address{J.M.S.: Institute of Mathematics at Toulouse, UMR CNRS 5219, Toulouse 3 University, 31062 Toulouse Cedex 9, France. {\tt schlenker@math.univ-toulouse.fr}}


\begin{abstract}
We study the integrals of type $I(a)=\int_{O_n}\prod u_{ij}^{a_{ij}}\,du$, depending on a matrix $a\in M_{p\times q}(\mathbb N)$, whose exact computation is an open problem. Our results are as follows: (1) an extension of the ``elementary expansion'' formula from the case $a\in M_{2\times q}(2\mathbb N)$ to the general case $a\in M_{p\times q}(\mathbb N)$, (2) the construction of the ``best algebraic normalization'' of $I(a)$, in the case $a\in M_{2\times q}(\mathbb N)$, (3) an explicit formula for $I(a)$, for diagonal matrices $a\in M_{3\times 3}(\mathbb N)$, (4) a modelling result in the case $a\in M_{1\times 2}(\mathbb N)$, in relation with the Euler-Rodrigues formula. Most proofs use various combinatorial techniques.  
\end{abstract}

\maketitle

\section*{Introduction}

An interesting open question, with several potential applications, is the exact computation of the polynomial integrals over the orthogonal group $O_n$. These integrals are best introduced in a ``rectangular form'', as functions of a matrix $a\in M_{p\times q}(\mathbb N)$, as follows:
$$I(a)=\int_{O_n}\prod_{i=1}^p\prod_{j=1}^qu_{ij}^{a_{ij}}\,du.$$

Of course, the case $p=q=n$ is the one we are interested in. The parameters $p,q$ are there only because they measure, and in a quite efficient way, the ``complexity'' of the problem. As an example, at $p=1$ one can use the basic fact that the first slice of $O_n$ is isomorphic to the sphere $S^{n-1}$, and the formula for the above integrals is simply
$$I\begin{pmatrix}a_1&\ldots&a_q\end{pmatrix}
=\varepsilon\cdot\frac{(n-1)!!a_1!!\ldots a_q!!}{(\Sigma a_i+n-1)!!},$$
where $m!!=(m-1)(m-3)(m-5)\ldots$, stopping at 1 or 2, and where $\varepsilon\in\{ 0,1\}$.

The complexity of the problem is measured as well by the value of $n\in\mathbb N$. For instance at $n=2$, a classical computation gives the following formula, with $\varepsilon\in\{-1,0,1\}$:
$$I_2\begin{pmatrix}a&c\\ b&d\end{pmatrix}=\varepsilon\cdot\frac{(a+d)!!(b+c)!!}{(a+b+c+d+1)!!}.$$

Another result, due to Diaconis and Shahshahani \cite{dsh}, is that the variables $n^{-1/2}u_{ij}$ are Gaussian and independent at $n=\infty$. This gives the following formula, with $\varepsilon\in\{0,1\}$:
$$I(a)=n^{-k}\left(\varepsilon\cdot\prod_{ij} a_{ij}!!+O(n^{-1})\right).$$
All these computations, basically done by using geometric methods, will be explained in detail in Section 1 below. We refer as well to Section 1 for the precise values of $\varepsilon$.  Some other geometric methods were developed by Olshanskii \cite{ols} and Neretin \cite{ner}. 

The integrals $I(a)$ can be studied as well by using combinatorial methods. It was known since Brauer \cite{bra} that the intertwining space ${\rm End}(u^{\otimes k})$ is spanned by certain explicit operators, associated to the pairings between $2k$ points. On the other hand, by basic representation theory, the orthogonal projection onto ${\rm End}(u^{\otimes k})$ is nothing but the matrix formed by the above quantities $I(a)$, with $\Sigma a_{ij}=2k$. So, Brauer's result reduces the computation of  $I(a)$ to certain combinatorial questions regarding pairings. The precise formula here, whose origins go back to Weingarten's paper \cite{wei}, is as follows:
$$\int_{O_n}u_{i_1j_1}\ldots u_{i_{2k}j_{2k}}\,du=\sum_{\pi,\sigma\in D_k}\delta_\pi(i)\delta_\sigma(j)W_{kn}(\pi,\sigma).$$
Here $D_k$ is the set of pairings of $\{1,\ldots,2k\}$, the delta symbols are $0$ or $1$, depending on whether the indices match or not, and $W_{kn}=G_{kn}^{-1}$, where $G_{kn}(\pi,\sigma)=n^{|\pi\vee\sigma|}$.

In general, the inversion of $G_{kn}$ is a difficult linear algebra problem. The first results here go back to Collins' paper \cite{col} for the unitary group $U_n$, and then to the paper of Collins and \'Sniady \cite{cs1}, where several other groups, including $O_n$, are discused. More results here were obtained by Collins and Matsumoto in \cite{cma}, and Zinn-Justin \cite{zin}. These combinatorial methods have a number of concrete applications, notably to random matrix questions \cite{sas}, \cite{cgm}, \cite{cs1}, \cite{cs2}, \cite{mno}, \cite{nov}, \cite{pre}, and to invariance questions \cite{def}.

The present paper is a continuation of our previous work \cite{int}. We will use the ``elementary expansion'' approach to the Weingarten function $W_{kn}$ developed there, in order to reach to some new results regarding the integrals $I(a)$. Our results are as follows:
\begin{enumerate}
\item We first have some theoretical results extending the ``elementary expansion'' techniques in \cite{int}, from the case $a\in M_{2\times q}(2\mathbb N)$ to the general case $a\in M_{p\times q}(\mathbb N)$.

\item As a first application of our improved elementary expansion methods we will construct the ``best algebraic normalization'' of $I(a)$, in the case $a\in M_{2\times q}(\mathbb N)$.

\item In the case where $a\in M_{p\times p}(\mathbb N)$ is diagonal the integrals $I(a)$ were explicitely computed in \cite{int} at $p=2$. Here we obtain an explicit formula at $p=3$.

\item Finally, we have a probabilistic modelling result for the pair of variables $(u_{11},u_{12})$, in relation with the Euler-Rodrigues formula, which is available at $n=3$. 
\end{enumerate}

The precise statements of the above results, as well as a substantial number of technical versions and comments regarding what happens in relation with each problem when looking at more general $p\times q$ shapes will be given in the body of the paper.

The paper is organized as follows: Section 1 is preliminary, and in Sections 2-3, 4-5, 6-7, 8-9 we prove the above results (1-4). Section 10 contains a few concluding remarks.

\subsection*{Acknowledgements}

We would like to thank B. Collins for several useful discussions. The work of T.B. was supported by the ANR grants ``Galoisint'' and ``Granma''.

\section{Group integrals}

Our basic object of study will be the polynomial integrals over $O_n$. It is convenient to introduce these integrals in a ``rectangular form'', as follows. 

\begin{definition}
Associated to any matrix $a\in M_{p\times q}(\mathbb N)$ is the integral
$$I(a)=\int_{O_n}\prod_{i=1}^p\prod_{j=1}^qu_{ij}^{a_{ij}}\,du$$
with respect to the uniform measure on the orthogonal group $O_n$.
\end{definition}

Our main tool for the computation of integrals over $O_n$ will be a combinatorial formula due to Collins and \'Sniady \cite{cs1}, whose origins go back to Weingarten's paper \cite{wei}.

\begin{theorem}
We have the Weingarten formula
$$\int_{O_n}u_{i_1j_1}\ldots u_{i_{2k}j_{2k}}\,du=\sum_{\pi,\sigma\in D_k}\delta_\pi(i)\delta_\sigma(j)W_{kn}(\pi,\sigma),$$
where the objects on the right are as follows:
\begin{enumerate}
\item $D_k$ is the set of pairings of $\{1,\ldots,2k\}$, also called Brauer diagrams.

\item The delta symbols are $0$ or $1$, depending on whether indices match or not.

\item The Weingarten matrix is $W_{kn}=G_{kn}^{-1}$, where $G_{kn}(\pi,\sigma)=n^{|\pi\vee\sigma|}$.
\end{enumerate}
\end{theorem}

As an example, the integrals of quantities of type $u_{i_1j_1}u_{i_2j_2}u_{i_3j_3}u_{i_4j_4}$ appear as sums of coefficients of the Weingarten matrix $W_{2n}$, which is given by:
$$W_{2n}=\begin{pmatrix}
n^2&n&n\\ 
n&n^2&n\\ 
n&n&n^2\end{pmatrix}^{-1}
=\frac{1}{n(n-1)(n+2)}
\begin{pmatrix}
n+1&-1&-1\\ 
-1&n+1&-1\\ 
-1&-1&n+1\end{pmatrix}.$$

More precisely, the various integrals at $k=2$ can be computed as follows:
\begin{eqnarray*}
I\begin{pmatrix}4&0\\ 0&0\end{pmatrix}
&=&\int_{O_n}u_{11}u_{11}u_{11}u_{11}\,du=\sum_{\pi,\sigma}W_{2n}(\pi,\sigma)=\frac{3}{n(n+2)},\\
I\begin{pmatrix}2&2\\ 0&0\end{pmatrix}
&=&\int_{O_n}u_{11}u_{11}u_{12}u_{12}\,du=\sum_\pi W_{2n}(\pi,\cap\cap)=\frac{1}{n(n+2)},\\
I\begin{pmatrix}2&0\\ 0&2\end{pmatrix}
&=&\int_{O_n}u_{11}u_{11}u_{22}u_{22}\,du=W_{2n}(\cap\cap,\cap\cap)=\frac{n+1}{n(n-1)(n+2)}.
\end{eqnarray*}

In general, the computation of the Weingarten matrix is a quite subtle combinatorial problem and the first results here go back to \cite{wei}, \cite{cs1}. A quite powerful formula was recently obtained in \cite{cma} and was further processed and clarified in \cite{zin}.

As a first application of the above Weingarten methods let us record a first basic result,   discussed in \cite{owg},  regarding the vanishing of $I(a)$, and its $n\to\infty$ behavior. We use the notation $m!!=(m-1)(m-3)(m-5)\ldots$, with the product stopping at 1 or 2.

\begin{proposition}
If the sum in some row or column of $a$ is odd, then $I(a)=0$. Also,
$$I(a)=n^{-k}\left(\varepsilon\cdot\prod_{ij}a_{ij}!!+O(n^{-1})\right),$$
where $k=\Sigma a_{ij}$, and $\varepsilon=1$ if all the entries are even, and $\varepsilon=0$ otherwise.
\end{proposition}

We can see from the above result that $I(a)$ depends in a non-trivial way on the parity of the entries $a_{ij}$. We have the following definition.

\begin{definition}
A matrix $a\in M_{p\times q}(\mathbb N)$ is called ``admissible'' if the sum on each of its rows and columns is an even number.
\end{definition}

In what follows we will restrict attention to the admissible matrices. We should mention here that we do not know if the converse of the vanishing statement in Proposition 1.3 holds or not. This is an interesting question, that we will not investigate in this paper.

Finally, let us work out a few concrete examples.

\begin{proposition}
At $n=2$ we have the formula
$$I\begin{pmatrix}a&c\\ b&d\end{pmatrix}=\varepsilon\cdot\frac{(a+d)!!(b+c)!!}{(a+b+c+d+1)!!},$$
where $\varepsilon=1$ if $a,b,c,d$ are even, $\varepsilon=-1$ is $a,b,c,d$ are odd, and $\varepsilon=0$ otherwise.
\end{proposition}

\begin{proof}
The idea is that we can restrict the integration to $SO_2=S^1$, and the result follows from a well-known trigonometric identity. See \cite{omx}.
\end{proof}

\begin{proposition}
For any $a\in\mathbb N^q$ we have
$$I(a)=\varepsilon\cdot\frac{(n-1)!!a_1!!\ldots a_q!!}{(\Sigma a_i+n-1)!!},$$
where $\varepsilon=1$ if all entries of $a$ are even, and $\varepsilon=0$ otherwise.
\end{proposition}

\begin{proof}
At $n=2$ this follows from Proposition 1.5. In general, the result follows from this observation by using spherical coordinates and the Fubini theorem. See \cite{omx}.
\end{proof}

\section{Elementary expansion}

We discuss here a general method from \cite{int} for the computation of the integrals $I(a)$. We will extend this method, from the case $a\in M_{2\times q}(2\mathbb N)$ discussed in \cite{int} to the general case, where $a\in M_{p\times q}(\mathbb N)$ is ``admissible'' in the sense of Definition 1.4.

We begin with a number of combinatorial definitions.

\begin{definition} 
Let $a_1,\ldots,a_p\in\mathbb N$ be such that the number $A=\Sigma a_k$ is even, and let $\sigma$ be a pairing of the set $\{1,\ldots,A\}$. We call ``type'' of $\sigma$ the matrix $r\in M_p(\mathbb N)$ where $r_{kl}$ is the number of elements of $\{(\sum_{i=1}^{k-1}a_i)+1,(\sum_{i=1}^{k-1}a_i)+2,\ldots,(\sum_{i=1}^{k-1}a_i)+a_k\}$ joined by $\sigma$ to elements of $\{(\sum_{i=1}^{l-1}a_i)+1,(\sum_{i=1}^{l-1}a_i)+2,\ldots,(\sum_{i=1}^{l-1}a_i)+a_l\}$.
\end{definition}

Observe that this matrix $r\in M_p(\mathbb N)$ is by definition symmetric, even on the diagonal, and the sum on its $i$-th row or column is equal to $a_i$, for any $i\in\{1,\ldots,p\}$.

As an example, let $p=3$ and $a_1=3,a_2=5,a_3=4$. In this case the total sum $A=12$ is indeed even, and the set $\{1,\ldots,A\}$ to be partitioned is:
$$\{1,\ldots,12\}=\{1,2,3,\underline{4},\underline{5},\underline{6},\underline{7},\underline{8},\overline{9},\overline{10},\overline{11},\overline{12}\}.$$
Here the upper and lower bars are there for reminding us that we have $12=3+5+4$. Now consider an arbitrary pairing of this set, as for instance the following one:
$$\sigma=\{(1,\underline{4}),(2,\underline{6}),(3,\overline{9}),(\underline{5},\underline{7}),(\underline{8},\overline{12}),(\overline{10},\overline{11})\}$$
According to the above definition, the type of this pairing is simply the symmetric $3\times 3$ matrix whose entries count the various types of paired points (underlined-underlined, underlined-overlined, and so on). So, the type of this pairing is:
$$r=\begin{pmatrix}
0&2&1\\
2&2&1\\
1&1&2
\end{pmatrix}.$$

Now back to the general case, we will need the following observation, which follows from definitions. The number of pairings of a given type $r\in M_p(\mathbb N)$ is given by:
$$K_r(a)=\left[\prod_{k<l}\begin{pmatrix}a_k\\ r_{kl}\end{pmatrix}\begin{pmatrix}a_l\\ r_{kl}\end{pmatrix}r_{kl}!\right]\left[\prod_ir_{ii}!!\right].$$

In what follows we will actually need an abstract formulation plus matrix generalization of the above notions. We have the following key definition. 

\begin{definition}
To any column vector $a\in M_{p\times 1}(\mathbb N)$ we associate the set $[a]$ consisting of the matrices $r\in M_p(\mathbb N)$ which are symmetric, even on the diagonal, and such that the sum on the $i$-th row and on the $i$-th column is equal to $a_i$. For $r\in[a]$ we let:
$$K_r(a)=\left[\prod_{k<l}\begin{pmatrix}a_k\\ r_{kl}\end{pmatrix}\begin{pmatrix}a_l\\ r_{kl}\end{pmatrix}r_{kl}!\right]\left[\prod_ir_{ii}!!\right].$$
More generally, to any matrix $a\in M_{p\times q}(\mathbb N)$ we associate the set $[a]=[a_1]\times\ldots\times[a_q]$, where $a_1,\ldots,a_q$ are the columns of $a$, and for $r\in[a]$ we let $K_r(a)=\prod_{j=1}^qK_{r^j}(a_j)$.
\end{definition}

We recall from \cite{int} that a matrix $a\in M_{p\times q}(\mathbb N)$ is called elementary when the sum on each of its columns is equal to 2. As pointed out there, all the Weingarten function values $W_{kn}(\pi,\sigma)$ are quantities of type $I(a)$, with $a$ being an elementary matrix. Moreover, it was shown in \cite{int} that a suitable modification of this idea leads to an ``elementary expansion'' formula for the two-row matrices. We refer the reader to \cite{int} for full details here, the point being that a quick reading of \cite{int} is probably useful, but definitely not necessary, for the understanding of the most general ``elementary expansion'' formula, to be explained now.

For $\alpha\in\mathbb R$ and $R\in\mathbb N$ we let $\alpha^R=\alpha\ldots \alpha$ ($R$ times), regarded as a $1\times R$ vector (that we call in the statement below ``block''). Also, for $r=(r^j)\in [a]$ we set $R_{kl}=\Sigma_jr_{kl}^j$. We have then the following result extending the similar ``two-row'' formula from \cite{int}.

\begin{proposition}
For any $a\in M_{p\times q}(\mathbb N)$ we have the ``elementary expansion'' formula
$$I(a)=\sum_{r\in[a]}K_r(a)I(E_r(a)),$$
where the elementary matrix $E_r(a)$ is constructed as follows:
\begin{enumerate}
\item There are $p$ rows, and the columns are block-indexed by $(k,l)$ with $k\leq l$.

\item On the $(k,l)$-th column $(k<l)$ we have an $1^{R_{kl}}$ block at rows $k$ and $l$.

\item On the $(i,i)$-th column we have a $2^{R_{ii}/2}$ block at row $i$.

\item The other entries are all equal to $0$.
\end{enumerate}
\end{proposition}

\begin{proof}
As a first remark, the above statement might look highly abstract, and its good understanding might seem to require a lot of preliminary examples. Indeed, so it is. In what follows we will simply present the abstract proof of the above statement, and we will leave the examples and particular cases for a bit later. These basically consist of the particular case $q=2$, discussed in Proposition 3.2 below, and of the particular case $q=3$, discussed in Proposition 6.1 below. The reader might find it useful to look first at Proposition 3.2 and Proposition 6.1 below, as an illustration of the above statement.

Let us first apply the Weingarten formula. With $A_i=\sum_{j=1}^qa_{ij}$, we get:
\begin{eqnarray*}
I\begin{pmatrix}a\end{pmatrix}
&=&\int_{O_n}u_{11}^{a_{11}}\ldots u_{1q}^{a_{1q}}\ldots\ldots u_{p1}^{a_{p1}}\ldots u_{pq}^{a_{pq}}\,du\\
&=&\sum_{\pi\sigma}\delta_\pi(1^{A_1}\ldots p^{A_p})\delta_\sigma(1^{a_{11}}\ldots q^{a_{1q}}\ldots\ldots1^{a_{p1}}\ldots q^{a_{pq}})W_{kn}(\pi,\sigma)\\
&=&\sum_\sigma\delta_\sigma(1^{a_{11}}\ldots q^{a_{1q}}\ldots\ldots1^{a_{p1}}\ldots q^{a_{pq}})\sum_\pi\delta_\pi(1^{A_1}\ldots p^{A_p})W_{kn}(\pi,\sigma).
\end{eqnarray*}
Now let us look at $\sigma$. In order for the $\delta_\sigma$ symbol not to vanish, $\sigma$ must connect between themselves the $\Sigma a_{i1}$ copies of $1$, the $\Sigma a_{i2}$ copies of 2, and so on, up to the $\Sigma a_{iq}$ copies of $q$. So let $r\in[a]$ be the type of $\sigma$ in the sense of Definition 2.2. If we denote by $[r]$ the set of pairings of type $r$, the above formula becomes:
$$I(a)=\sum_{r\in[a]}\sum_{\sigma\in[r]}\sum_\pi\delta_\pi(1^{A_1}\ldots p^{A_p})W_{kn}(\pi,\sigma).$$
Let us compute now the elementary integral in the above statement. We label by $\sigma_1,\ldots,\sigma_k$ the strings of $\sigma$, and we consider the multi-index $j\in\{1,\ldots,k\}^{pk}$ given by $j_r\in\sigma_r$, for any $r\in\{1,\ldots,k\}$, so that $\delta_{\sigma'}(j)=\delta_{\sigma\sigma'}$ for any $\sigma'$. From the Weingarten formula we get:
\begin{eqnarray*}
I(E_r(a))
&=&\int_{O_n}u_{1j_1}\ldots u_{1j_{A_1}}\ldots\ldots u_{pj_{A_1+\ldots+A_{p-1}+1}}\ldots u_{pj_{A_1+\ldots+A_p}}\,du\\
&=&\sum_{\pi\sigma'}\delta_\pi(1^A2^B)\delta_{\sigma'}(j)W_{kn}(\pi,\sigma')\\
&=&\sum_\pi\delta_\pi(1^{A_1}\ldots p^{A_p})W_{kn}(\pi,\sigma).
\end{eqnarray*}
Together with the above formula and with $K_r(a)=\# [r]$ this finishes the proof.
\end{proof}

As a first consequence of the elementary expansion formula we have the following key result, which already appeared in \cite{int} in a more particular formulation.

\begin{proposition}
We have the ``compression formula''
$$I\begin{pmatrix}a&c\\ b&0\end{pmatrix}=\frac{\prod c_j!!}{(\Sigma c_j)!!}
\,I\begin{pmatrix}a&\Sigma c_j\\ b&0\end{pmatrix}$$
valid for any vectors $a\in\mathbb N^q,c\in(2\mathbb N)^r$ and any matrix $b\in M_{p\times q}(\mathbb N)$.
\end{proposition}

\begin{proof}
Consider the elementary expansion formula for the matrix in the statement. By splitting apart the indices corresponding to the $c$ columns, we get:
\begin{eqnarray*}
I\begin{pmatrix}a&c\\ b&0\end{pmatrix}
&=&\sum_{r\in[^a_b\!\!{\ }^c_0]}K_r\begin{pmatrix}a&c\\ b&0\end{pmatrix}I\left(E_r\begin{pmatrix}a&c\\ b&0\end{pmatrix}\right)\\
&=&\sum_{r\in[^a_b]}\sum_{s\in[^c_0]}K_r\begin{pmatrix}a\\ b\end{pmatrix}K_s\begin{pmatrix}c\\ 0\end{pmatrix}I\left(E_{rs}\begin{pmatrix}a&c\\ b&0\end{pmatrix}\right).
\end{eqnarray*}
Let us look now at the indices $s$. These are by definition the $r$-tuples $s=(s^1,\ldots,s^r)$ with $s^j\in [^{c_j}_{\,0}]$ for any $i$. Now since there is obviously ``no combinatorics'' when trying to pair $c_j+0+\ldots+0$ elements, the only allowed values are $s^j={\rm diag}(c_j,0,\ldots,0)$, and the corresponding coefficients are $K_{s_j}(^{c_j}_{\,0})=c_j!!$. So, with this value of $s=(s^j)$, we get:
$$I\begin{pmatrix}a&c\\ b&0\end{pmatrix}=\prod_jc_j!!\sum_{r\in[^a_b]}K_r\begin{pmatrix}a\\ b\end{pmatrix}I\left(E_{rs}\begin{pmatrix}a&c\\ b&0\end{pmatrix}\right).$$
Let us look now at the matrix on the right. According to the definitions, this matrix is obtained from $E_r(^a_b)$ by ``fattening'' it with the $S_{kl}=\Sigma_js^j_{kl}$ parameters. Now since these parameters are simply given by $S={\rm diag}(\Sigma c_j,0,\ldots,0)$, we conclude that the elementary matrix on the right depends only on $\Sigma c_j$, so we have a formula of the following type:
$$I\begin{pmatrix}a&c\\ b&0\end{pmatrix}=\prod_jc_j!!\sum_{r\in[^a_b]}K_r\begin{pmatrix}a\\ b\end{pmatrix}J_{r,\Sigma c_j}\begin{pmatrix}a\\ b\end{pmatrix}.$$
We can apply now this formula with $\Sigma c_j$ at the place of $c$, and we get:
$$I\begin{pmatrix}a&\Sigma c_j\\ b&0\end{pmatrix}=\left(\sum_jc_j\right)!!\sum_{r\in[^a_b]}K_r\begin{pmatrix}a\\ b\end{pmatrix}J_{r,\Sigma c_j}\begin{pmatrix}a\\ b\end{pmatrix}.$$
Now by comparing the above two formulae we get the formula in the statement.
\end{proof}

The above results suggest the following normalization.

\begin{definition}
We make the normalization
$$I(a)=\frac{(n-1)!!\prod a_{ij}!!}{(\Sigma a_{ij}+n-1)!!}\,J(a),$$
where $I$ denotes as usual the polynomial integration over $O_n$.
\end{definition}

Let us rewrite the results that we have so far in terms of this new quantity.

\begin{proposition}
The function $J$ has the following properties:
\begin{enumerate}
\item Invariance under row and column permutation.

\item Invariance under matrix transposition.

\item Invariance under horizontal and vertical compression.

\item Triviality on the row and column matrices.
\end{enumerate}
\end{proposition}

\begin{proof}
The assertions (1) and (2) are clear from definitions, and (3), (4) are just reformulations of Proposition 2.4 and Proposition 1.6, by using the invariance under transposition. The ``triviality'' condition in (4) means of course that we have $J\in\{0,1\}$. 
\end{proof}

We let $K_r'(a)=K_r(a)/(\prod a_i!!)$ for vectors, and $K_r'(a)=\prod K_r'(a_j)$ in general.

\begin{theorem}
We have the ``compressed elementary expansion'' formula
$$J(a)=\sum_{r\in[a]}K_r'(a)J(E_r'(a)),$$
where the matrix $E_r'(a)$ is constructed as follows:
\begin{enumerate}
\item There are $p$ rows, and the columns are block-indexed by $(k,l)$ with $k\leq l$.

\item On the $(k,l)$-th column $(k<l)$ we have an $1^{R_{kl}}$ block at rows $k$ and $l$.

\item On the $(i,i)$-th column we have an $R_{ii}$ entry at row $i$.

\item The other entries are all equal to $0$.
\end{enumerate}
\end{theorem}

\begin{proof}
This is just a reformulation of Proposition 2.2, by using the compression formula. Indeed, consider the elementary expansion formula for the matrix in the above statement:
$$I(a)=\sum_{r\in[a]}K_r(a)I(E_r(a)).$$
If we look at the definition of $E_r(a)$, we see that for each of these matrices the quantity $\Sigma a_{ij}$ is the same as for $a$, and the quantity $\prod a_{ij}$ is simply equal to $1$. Thus when translating the above formula in terms of $J$ quantities, the only thing that happens is that the right term gets divided by $\prod a_{ij}!!$. This normalization factor can be included into the quantities $K_r(a)$, and gives rise to modified quantities $K'_r(a)$ which appear in the above statement:
$$J(a)=\sum_{r\in[a]}K_r'(a)J(E_r(a)).$$
Let us look now at the matrices $E_r(a)$ appearing on the right. By using the compression formula for each of them, on all rows, we have $J(E_r(a))=J(E_r'(a))$, and this finishes the proof.
\end{proof}

We discuss now some statements which reformulate and generalize the results that are left, namely Proposition 1.5 and Proposition 1.6. We use the following notation.

\begin{definition}
To any $a_i,b_i\in\mathbb N$ and $n\in\mathbb N$ we associate the following quantity:
$$\left[\frac{a_1\ldots a_k}{b_1\ldots b_s}\right]=\frac{(a_1+n-2)!!\ldots (a_k+n-2)!!}{(b_1+n-2)!!\ldots(b_s+n-2)!!}.$$
This quantity will be called ``balanced'' if $k=s$ and $\Sigma a_i=\Sigma b_i$.
\end{definition}

We already know that we have $J=\varepsilon$ on one-row vectors, and this reformulates Proposition 1.6. The following statement is a key generalization of this fact.

\begin{proposition}
We have the ``cross formula''
$$J\begin{pmatrix}0&b_1&0\\ c_1&a&c_2\\ 0&b_2&0\end{pmatrix}
=\varepsilon\left[\frac{0(B+C)}{BC}\right],$$
where $B=\Sigma b_i,C=\Sigma c_i$, and $\varepsilon=1$ if all entries are even, and $\varepsilon=0$ otherwise.
\end{proposition}

\begin{proof}
By doing two compressions, it is enough to check the following formula:
$$J\begin{pmatrix}a&c\\ b&0\end{pmatrix}
=\varepsilon\left[\frac{0(b+c)}{bc}\right].$$
We proceed by induction, with step 2. At $c=0,1$ the assertion follows from the 1-row formula. In general, by using the trivial identity $\Sigma u_{1i}^2=1$, we get:
$$I\begin{pmatrix}a&c\\ b&0\end{pmatrix}
=I\begin{pmatrix}a+2&c\\ b&0\end{pmatrix}
+I\begin{pmatrix}a&c+2\\ b&0\end{pmatrix}
+(n-2)I\begin{pmatrix}a&c&2\\ b&0&0\end{pmatrix}.$$
In terms of $J$ quantities, and after doing a compression, we get:
$$(a+b+c+n)J\begin{pmatrix}a&c\\ b&0\end{pmatrix}
=(c+n-1)J\begin{pmatrix}a&c+2\\ b&0\end{pmatrix}
+(a+1)J\begin{pmatrix}a+2&c\\ b&0\end{pmatrix}.$$
Now since by induction the left and right $J$ quantities are equal, we get:
$$J\begin{pmatrix}a&c+2\\ b&0\end{pmatrix}
=\frac{b+c+n-1}{c+n-1}J\begin{pmatrix}a&c\\ b&0\end{pmatrix}.$$
Thus the formula to be proved is true at $c+2$, which finishes the proof.
\end{proof}

Finally, Proposition 1.5 has the following reformulation.

\begin{proposition}
At $n=2$ we have the formula
$$J\begin{pmatrix}a&c\\ b&d\end{pmatrix}
=\varepsilon\left[\frac{0^2(a+d)(b+c)}{abcd}\right]$$
where $\varepsilon=1$ if $a,b,c,d$ are even, $\varepsilon=-1$ if $a,b,c,d$ are odd, and $\varepsilon=0$ otherwise. 
\end{proposition}

\begin{proof}
This is indeed just a reformulation of Proposition 1.5.
\end{proof}

\section{The two-row case}

In this section we restrict attention to the two-row case, and we extend some technical results from \cite{int}, from the ``even'' case $a\in M_{2\times q}(2\mathbb N)$ considered there, to the general case $a\in M_{2\times q}(\mathbb N)$. These extensions will be of great use in Sections 4-5 below.

We first replace Definition 2.2 by a simpler definition, adapted to the two-row case. 

\begin{definition}
We let $[a,b]=\{m,m-2,m-4,\ldots\}$, with the set stopping at $0$ or $1$, where $m=\min(a,b)$. If $a,b$ do not have the same parity, we set $[a,b]=\emptyset$.
\end{definition}

Observe that this definition agrees with Definition 2.2 in the $2\times 1$ case. Indeed, in both situations an element $r\in [a,b]$ is just the type of a pairing of $a+b$ elements. For more explanations here, we refer the reader to the proof of Proposition 3.2 below.

\begin{proposition}
We have the ``two-row elementary expansion'' formula
$$J\begin{pmatrix}a\\ b\end{pmatrix}
=\sum_{r_1\ldots r_q}K_r'\begin{pmatrix}a\\ b\end{pmatrix}
J\begin{pmatrix}1^R&A-R&0\\ 1^R&0&B-R\end{pmatrix},$$
where the sum is over $r_j\in [a_j,b_j]$, and $A=\Sigma a_i,B=\Sigma b_i,R=\Sigma r_i$.
\end{proposition}

\begin{proof}
We apply Theorem 2.7 to the matrix in the statement:
$$J\begin{pmatrix}a\\ b\end{pmatrix}
=\sum_rK_r'\begin{pmatrix}a\\ b\end{pmatrix}
J\left(E_r'\begin{pmatrix}a\\ b\end{pmatrix}\right).$$
According to the above new conventions, in this formula the indices range indeed as stated, $r_j\in[a_j,b_j]$. In order to compute the matrix on the right, let us observe that in terms of the new indices $r_j\in[a_j,b_j]$, the matrix indices in Definition 2.2 are:
$$r_j^{old}=\begin{pmatrix}a_j-r_j&r_j\\ r_j&b_j-r_j\end{pmatrix}.$$
By plugging these values into the definition of $E_r'$, we obtain:
$$E_r'\begin{pmatrix}a\\ b\end{pmatrix}
=\begin{pmatrix}1^R&A-R&0\\ 1^R&0&B-R\end{pmatrix}.$$
This gives the formula in the statement.
\end{proof}

We denote by $A,B,C,\ldots$ the sum of entries of our various matrices $a,b,c,\ldots$

\begin{theorem}
The function $F$ defined on two-row matrices by
$$F\begin{pmatrix}a\\ b\end{pmatrix}
=\left[\frac{AB}{0(A+B)}\right]J\begin{pmatrix}a\\ b\end{pmatrix}$$
satisfies the ``flipping formula'' $F(^a_b{\ }^c_d)=F(^a_b{\ }^d_c)$ for any vectors $a,b,c,d$.
\end{theorem}

\begin{proof}
We apply the elementary expansion formula to the matrices in the statement. Since the passage from $J$ to $F$ depends only on the sums on the rows, which are the same for the matrix $a$ and for its versions $E_r'(a)$, we can replace $J$ by $F$, and we get:
$$F\begin{pmatrix}a&c\\ b&d\end{pmatrix}
=\sum_{r_is_j}\prod_{ij}K_{r_i}'\begin{pmatrix}a_i\\ b_i\end{pmatrix}K_{s_j}'\begin{pmatrix}c_j\\ d_j\end{pmatrix}F\begin{pmatrix}1^{R+S}&A+C-R-S&0\\
1^{R+S}&0&B+D-R-S\end{pmatrix},$$
$$F\begin{pmatrix}a&d\\ b&c\end{pmatrix}
=\sum_{r_is_j}\prod_{ij}K_{r_i}'\begin{pmatrix}a_i\\ b_i\end{pmatrix}K_{s_j}'\begin{pmatrix}d_j\\ c_j\end{pmatrix}F\begin{pmatrix}1^{R+S}&A+D-R-S&0\\
1^{R+S}&0&B+C-R-S\end{pmatrix}.$$
Our claim is that the two formulae are in fact identical. Indeed, the first remark is that the various indices vary in the same sets. Also, since $K_r'(^a_b)$ is symmetric in $a,b$, the coefficients are the same. So, it remains to prove that the $F$ quantities on the right are equal. And here, by changing the notations, the statement is that for any $s\in\mathbb N$ and any $a,b,c,d\in\mathbb N$ satisfying $a+b=c+d$, we have:
$$F\begin{pmatrix}1^s&2a&0\\ 1^s&0&2b\end{pmatrix}
=F\begin{pmatrix}1^s&2c&0\\ 1^s&0&2d\end{pmatrix}.$$
In order to prove this result, we use the cross formula, which gives:
$$F\begin{pmatrix}a&c\\ b&0\end{pmatrix}
=\varepsilon\left[\frac{(a+c)(b+c)}{c(a+b+c)}\right].$$
We prove the result by induction over $s$. At $s=0$ the dependence on $a+b$ follows from the cross formula. So, assume that the result is true at $s\in\mathbb N$. We have:
$$F\begin{pmatrix}a&c\\ b&0\end{pmatrix}
=\sum_{r\in[a,b]}K_r'\begin{pmatrix}a\\ b\end{pmatrix}
F\begin{pmatrix}1^r&a+c-r&0\\ 1^r&0&b-r\end{pmatrix},$$
$$F\begin{pmatrix}a&0\\ b&c\end{pmatrix}
=\sum_{r\in[a,b]}K_r'\begin{pmatrix}a\\ b\end{pmatrix}
F\begin{pmatrix}1^r&a-r&0\\ 1^r&0&b+c-r\end{pmatrix}.$$
We know from the cross formula that the terms on the left, and hence the sums on the right, are equal. With $b=s$, this equality becomes:
$$\sum_{r\in[a,s]}K_r'\begin{pmatrix}a\\ s\end{pmatrix}
F\begin{pmatrix}1^r&a+c-r&0\\ 1^r&0&b-r\end{pmatrix}
=\sum_{r\in[a,s]}K_r'\begin{pmatrix}a\\ s\end{pmatrix}
F\begin{pmatrix}1^r&a-r&0\\ 1^r&0&b+c-r\end{pmatrix}.$$
Now if we choose $a\geq s$, the sums range over the set $[a,s]=\{s,s-2,s-4\ldots\}$, and by the induction assumption, all the terms starting from the second one coincide. Thus, the above equality tells us that the first terms ($r=s$) of the two sums are equal:
$$F\begin{pmatrix}1^s&a+c-s\\ 1^s&0\end{pmatrix}
=F\begin{pmatrix}1^s&a-s&0\\ 1^s&0&c\end{pmatrix}.$$
Since this equality holds for any $a\geq s$ and for any $c$, this shows that the elementary flipping formula holds at $s$, and this finishes the proof.
\end{proof}

\begin{proposition}
The quantities in the two-row elementary expansion formula are given by
$$K_r'\begin{pmatrix}a\\ b\end{pmatrix}
=\prod_j\frac{(a_j+1)!!(b_j+1)!!}{r_j!(a_j-r_j+1)!!(b_j-r_j+1)!!},$$
$$J\begin{pmatrix}1^R&A-R&0\\ 1^R&0&B-R\end{pmatrix}
=(-1)^{R/2}R!!\left[\frac{0(A+B-R)}{AB}\right],$$
where the indices run as usual, $r_j\in [a_j,b_j]$, and $A=\Sigma a_i,B=\Sigma b_i,R=\Sigma r_i$.
\end{proposition}

\begin{proof}
We use the notations from Proposition 3.2 and its proof. According to the identifications made in the beginning of this section, the coefficients are:
\begin{eqnarray*}
K_r'\begin{pmatrix}a\\ b\end{pmatrix}
&=&\prod_j\frac{1}{a_j!!b_j!!}\begin{pmatrix}a_j\\ r_j\end{pmatrix}\begin{pmatrix}b_j\\ r_j\end{pmatrix}r_j!(a_j-r_j)!!(b_j-r_j)!!\\
&=&\prod_j\frac{a_j!b_j!r_j!(a_j-r_j)!!(b_j-r_j)!!}{a_j!!b_j!!r_j!(a_j-r_j)!r_j!(b_j-r_j)!}\\
&=&\prod_j\frac{(a_j+1)!!(b_j+1)!!}{r_j!(a_j-r_j+1)!!(b_j-r_j+1)!!}.
\end{eqnarray*}
As for the values of $J$, these can be computed by a method presented in \cite{int}. The idea is that, by using several times the flipping principle, these values can be deduced from the system of equations coming from the elementary expansion of the quantities $J(^a_b)=1$.
\end{proof}

\begin{proposition}
We have the ``transmutation formula''
$$J_n\begin{pmatrix}a&c\\ b&0\end{pmatrix}
=\left[\frac{0(B+C)}{BC}\right]_n\,
J_{n+C}\begin{pmatrix}a\\ b\end{pmatrix}$$
valid for any vectors $a,b\in\mathbb N^q$ and $c\in\mathbb N^s$.
\end{proposition}

\begin{proof}
According to Proposition 3.2 and Proposition 3.4, we have:
\begin{eqnarray*}
J_n\begin{pmatrix}a&c\\ b&0\end{pmatrix}
&=&\sum_rK_r'\begin{pmatrix}a\\ b\end{pmatrix}
(-1)^{R/2}R!!\left[\frac{0(A+B+C-R)}{(A+C)B}\right]_n,\\
J_{n+C}\begin{pmatrix}a\\ b\end{pmatrix}
&=&\sum_rK_r'\begin{pmatrix}a\\ b\end{pmatrix}
(-1)^{R/2}R!!\left[\frac{0(A+B-R)}{AB}\right]_{n+C}.
\end{eqnarray*}
Now observe that the quantities on the right are related by:
\begin{eqnarray*}
\left[\frac{0(A+B+C-R)}{(A+C)B}\right]_n
&=&\left[\frac{0(B+C)}{BC}\right]_n\cdot \left[\frac{C(A+B+C-R)}{(A+C)(B+C)}\right]_n\\
&=&\left[\frac{0(B+C)}{BC}\right]_n\cdot \left[\frac{0(A+B-R)}{AB}\right]_{n+C}.
\end{eqnarray*}
Since the other quantities in the above two sums are identical, this gives the result.
\end{proof}

Finally, let us record the following useful result concerning some supplementary matrices for which $I$ decomposes as a product of factorials.

\begin{proposition}
We have the ``spark formula''
$$J\begin{pmatrix}x&a&0\\ y&0&b\end{pmatrix}
=\varepsilon\left[\frac{0(a+b+x)(a+b+y)}{(a+b)(a+x)(b+y)}\right]$$
where $\varepsilon=1$ if $a,b,x,y$ are even, and $\varepsilon=0$ otherwise.
\end{proposition}

\begin{proof}
The case $\varepsilon=0$ is clear. In the case $\varepsilon=1$ we use the flipping principle, which tells us that the following quantities are equal:
\begin{eqnarray*}
F\begin{pmatrix}x&a&0\\ y&0&b\end{pmatrix}
&=&\left[\frac{(a+x)(b+y)}{0(a+b+x+y)}\right]
J\begin{pmatrix}x&a&0\\ y&0&b\end{pmatrix},\\
F\begin{pmatrix}x&0&0\\ y&a&b\end{pmatrix}
&=&\left[\frac{x(a+b+y)}{0(a+b+x+y)}\right]
J\begin{pmatrix}x&0&0\\ y&a&b\end{pmatrix}.
\end{eqnarray*}
In order to compute the last $J$ quantity we can use the cross formula, and we get:
\begin{eqnarray*}
J\begin{pmatrix}x&a&0\\ y&0&b\end{pmatrix}
&=&\left[\frac{x(a+b+y)}{(a+x)(b+y)}\right]J\begin{pmatrix}x&0&0\\ y&a&b\end{pmatrix}\\
&=&\left[\frac{x(a+b+y)}{(a+x)(b+y)}\right]\cdot\left[\frac{0(a+b+x)}{x(a+b)}\right].
\end{eqnarray*}
After simplifying the fraction, this gives the formula in the statement.
\end{proof}

\section{The normalization problem}

We discuss here the normalization problematics for the quantity $J(a)$ for matrices $a\in M_{2\times q}(\mathbb N)$ with $q=2,3$. The general case will be discussed in the next section.

We begin with the case $q=2$, where the answer is particularly simple.

\begin{proposition}
The function $\varphi$ defined by
$$J\begin{pmatrix}a&c\\ b&d\end{pmatrix}
=\left[\frac{0^2(a+d)(b+c)}{abcd}\right]
\varphi\begin{pmatrix}a&c\\ b&d\end{pmatrix}$$
is symmetric in $a,b,c,d$, and is trivial on crosses, and at $n=2$.
\end{proposition}

\begin{proof}
Since both $J$ and the above normalization constant and invariant under row and column permutation, and under transposition, so is the function $\varphi$. In addition, we know from the flipping principle that we have:
$$\left[\frac{(a+c)(b+d)}{0(a+b+c+d)}\right]J\begin{pmatrix}a&c\\ b&d\end{pmatrix}
=\left[\frac{(a+d)(b+c)}{0(a+b+c+d)}\right]J\begin{pmatrix}a&d\\ b&c\end{pmatrix}.$$
This formula tells us that $\varphi$ is as well invariant under flliping of columns, so we can conclude that $\varphi$ is symmetric in $a,b,c,d$. For the cross assertion now, we have:
$$J\begin{pmatrix}a&c\\ b&0\end{pmatrix}
=\left[\frac{0(b+c)}{bc}\right]
\varphi\begin{pmatrix}a&c\\ b&0\end{pmatrix}.$$
By comparing with the cross formula, we obtain that we have indeed $\varphi=\varepsilon$ in this case. Finally, the triviality at $n=2$ follows from Proposition 2.6.
\end{proof}

Our objective now is to find an abstract characterization of the above quantity $\varphi$. The methods to be developed below will be of great use for the study at $q\geq 3$.

\begin{definition}
A formula of type $J(a)=\Gamma(a)\varphi(a)$ is called ``normalization'' if:
\begin{enumerate}
\item $\Gamma=[S_1\ldots S_k/T_1\ldots T_k]$, where $S_i,T_i$ are sums of entries of $a$, with $\Sigma S_i=\Sigma T_i$.

\item $\varphi$ is invariant under column permutation and flip, and is trivial on crosses.
\end{enumerate}
\end{definition}

As a first example, Proposition 4.1 solves the normalization problem at $q=2$.

The above definition probably deserves some more comments. The point is that the invariance of a function $\varphi:M_{2\times q}(\mathbb N)\to\mathbb R$ under column permutation and flip, i.e. what we can expect from a ``careful normalization'' of $J:M_{2\times q}(\mathbb N)\to\mathbb R$, just means invariance under the hyperoctahedral group $H_q\subset S_{2q}$. At $q=2$ we have the invariance under the whole group $S_{2q}$, simply because we are in a very special situation, where the transposition is available. In the general case $q\in\mathbb N$, we cannot expect $\varphi$ to be invariant under $S_{2q}$.

We can reformulate the normalization condition in terms of $\Gamma$ only. We denote as usual by $A,B,C,\ldots$ the sum of entries of the various row vectors $a,b,c,\ldots$ 

\begin{proposition}
$\Gamma=[S_1\ldots S_k/T_1\ldots T_k]$ is a normalization constant if and only if:
\begin{enumerate}
\item $\Lambda(^a_b)=[AB/0(A+B)]\Gamma(^a_b)$ is invariant under column permutation and flip.

\item $\Lambda(^a_b{\ }^c_0)=[(a+C)(b+C)/C(a+b+C)]$, for any $a,b\in 2\mathbb N$ and $c\in(2\mathbb N)^{q-1}$.
\end{enumerate}
\end{proposition}

\begin{proof}
We know from Proposition 2.6 that $J$ is invariant under column permutation, so the invariance of $\varphi$ under column permutation is equivalent to that of $\Gamma$, or of $\Lambda$. We also know from Theorem 3.3 that $[AB/0(A+B)]J(^a_b)$ is invariant under column flip, so the invariance of $\varphi$ under column flip is equivalent to that of $\Lambda$. Finally, for the checking the triviality on crosses we can restrict attention to the case where all the entries are even. And here by Proposition 2.9 we have $J(^a_b{\ }^c_0)=[0(b+C)/bC]$, so $\varphi$ is trivial if and only if:
$$\Lambda\begin{pmatrix}a&c\\ b&0\end{pmatrix}
=\left[\frac{(a+C)b}{0(a+b+C)}\right]
\cdot\left[\frac{0(b+C)}{bC}\right].$$
This gives the formula in the statement, and this finishes the proof.
\end{proof}

\begin{theorem}
With the above conventions, the normalization problem has a unique solution for the $2\times 2$ matrices, namely the one given by Proposition 4.1.
\end{theorem}

\begin{proof}
Let $\Gamma(^a_b{\ }^c_d)=[S_1\ldots S_k/T_1\ldots T_k]$ be a normalization constant. Here $S_i,T_i$ are certain sums of entries of $(^a_b{\ }^c_d)$. There are $2^4=16$ such possible sums, and under the action of the column permutations and flips, these sums split into 6 classes, as follows:
\begin{eqnarray*}
(00)&=&0,\\
(10)&=&abcd,\\
(11)&=&(a+c)(b+d)(a+d)(b+c),\\
(20)&=&(a+b)(c+d),\\
(21)&=&(a+b+c)(a+b+d)(a+c+d)(b+c+d),\\
(22)&=&a+b+c+d.
\end{eqnarray*}
Here the symbols on the left, which are all of type $(xy)$ with $2\geq x\geq y\geq 0$, mean that we select $x$ entries from the first column, and $y$ entries from the second column. The expressions on the right are the products of all sums obtained by using the symbols on the left, for the matrix itself, and for its conjugate by column permutation.

Proposition 4.3 (1) tells us that we must have a formula of the following type:
$$\Lambda=\left[(00)^A(10)^B(11)^C(20)^D(21)^E(22)^F\right].$$
We also know that $\Lambda$ must be ``balanced'' in the sense of Definition 2.8. Thus we have the following ``balancing equations'', for the total number of terms, and the total sum:
\begin{eqnarray*}
\#&:&A+4B+4C+2D+4E+F=0,\\
\Sigma&:&B+2C+D+3E+F=0.
\end{eqnarray*}
Let us look now at the last normalization requirement, namely the triviality on crosses. By plugging the value $d=0$ into the above expressions, these become:
\begin{eqnarray*}
(00)&=&0,\\
(10)&=&0abc,\\
(11)&=&ab(a+c)(b+c),\\
(20)&=&c(a+b),\\
(21)&=&(a+b+c)(a+b)(a+c)(b+c),\\
(22)&=&a+b+c.
\end{eqnarray*}
According to Proposition 4.3 (2), the triviality on crosses is equivalent to:
$$\left[(00)^A(10)^B(11)^C(20)^D(21)^E(22)^F\right]=\left[\frac{(a+c)(b+c)}{c(a+b+c)}\right].$$
By identifying coefficients, we obtain the following supplementary equations:
\begin{eqnarray*}
0&:&A+B=0,\\
a&:&B+C=0,\\
c&:&B+D=-1,\\
a+b&:&D+E=0,\\
a+c&:&C+E=1,\\
a+b+c&:&E+F=-1.
\end{eqnarray*}
The unique solution of the system is $A=C=1,B=F=-1,D=E=0$, so:
$$\Lambda
=\left[\frac{(00)(11)}{(10)(22)}\right]
=\left[\frac{0(a+c)(b+d)(a+d)(b+c)}{abcd(a+b+c+d)}\right].$$
Thus the normalization constant is:
$$\Gamma
=\left[\frac{0(a+b+c+d)}{(a+c)(b+d)}\right]
\cdot\left[\frac{0(a+c)(b+d)(a+d)(b+c)}{abcd(a+b+c+d)}\right]
=\left[\frac{0^2(a+d)(b+c)}{abcd}\right].$$
But this is exactly the constant in Proposition 4.1, and this finishes the proof.
\end{proof}

We have the following technical extension of Theorem 4.4.

\begin{proposition}
The formula
$$J\begin{pmatrix}a&b&c\\ x&y&z\end{pmatrix}=\left[\frac{\begin{matrix}abc(a+b+z)(a+c+y)(b+c+x)\\ xyz(a+y+z)(b+x+z)(c+x+y)\end{matrix}}{\begin{matrix}(a+b)(a+c)(b+c)(x+y)(x+z)(y+z)\\ (a+y)(a+z)(b+x)(b+z)(c+x)(c+y)\end{matrix}}\right]\varphi\begin{pmatrix}a&b&c\\ x&y&z\end{pmatrix}$$
is the unique solution of the normalization problem, for the $2\times 3$ matrices.
\end{proposition}

\begin{proof}
We follow the method in the proof of Theorem 4.4, with less explanations, and with some shortcuts. There are 10 symbols of type $(xyz)$ with $2\geq x\geq y\geq z\geq 0$, and Proposition 4.3 (1) tells us that we must have a formula of the following type:
$$\Lambda=\left[(000)^A(100)^B(110)^C(111)^D(200)^E(210)^F(211)^G(220)^H(221)^I(222)^J\right].$$
The ``balancing'' equations are as follows:
\begin{eqnarray*}
\#&:&A+6B+12C+8D+3E+12F+12G+3H+6I+J=0,\\
\Sigma&:&B+4C+4D+E+6F+8G+2H+5I+J=0.
\end{eqnarray*}
The values of the basic products on a cross ($y=z=0$) are:
\begin{eqnarray*}
(000)&=&0,\\
(100)&=&0^2abcx,\\
(110)&=&0a^2bcx^2(a+b)(a+c)(b+c)(b+x)(c+x),\\
(111)&=&ax(a+b)(a+c)(b+x)(c+x)(a+b+c)(b+c+x),\\
(200)&=&bc(a+x),\\
(210)&=&bc(a+b)(a+c)(b+c)^2(a+x)^2(b+x)(c+x)(a+b+x)(a+c+x),\\
(211)&=&(a+b)(a+c)(a+x)(b+x)(c+x)(a+b+c)^2\\
&&(a+b+x)(a+c+x)(b+c+x)^2(a+b+c+x),\\
(220)&=&(b+c)(a+b+x)(a+c+x),\\
(221)&=&(a+b+x)(a+c+x)(b+c+x)(a+b+c)(a+b+c+x)^2,\\
(222)&=&a+b+c+x.
\end{eqnarray*}
According to Proposition 4.3 (2), we must have the following formula:
$$\left[(000)^A(100)^B\ldots(222)^J\right]
=\left[\frac{(a+b+c)(b+c+x)}{(b+c)(a+b+c+x)}\right].$$
We obtain the following supplementary equations:
\begin{eqnarray*}
0&:&A+2B+C=0,\\
a&:&B+2C+D=0,\\
b&:&B+C+E+F=0,\\
a+b&:&C+D+F+G=0,\\
b+c&:&C+2F+H=-1,\\
a+x&:&E+2F+G=0,\\
a+b+c&:&D+2G+I=1,\\
a+b+x&:&F+G+H+I=0,\\
a+b+c+x&:&G+2I+J=-1.
\end{eqnarray*}
The unique solution is $A=C=J=-1,B=D=1,E=F=G=H=I=0$, so:
$$\Lambda=\left[\frac{(100)(111)}{(000)(110)(222)}\right].$$
Thus the normalization constant is unique, given by:
$$\Gamma
=\left[\frac{0(A+B)}{AB}\right]\cdot\left[\frac{(100)(111)}{(000)(110)(222)}\right]
=\left[\frac{(100)((111)/AB)}{(110)}\right].$$
But this is the constant in the statement, and this finishes the proof.
\end{proof}

\section{Two-row normalization}

In this section we solve the normalization problem for the arbitrary two-row matrices. Let us first introduce a candidate for the normalization constant.

\begin{definition}
To any matrix $a\in M_{2\times q}(\mathbb N)$ and any $r\leq q$ we associate the product
$$P_r=\prod_S\prod_{s\in S}\prod_{i_s\in\{1,2\}}\left(\sum_{t\in S}a_{i_tt}\right)$$ 
with $S$ ranging over all subsets of $\{1,\ldots,q\}$, having $r$ elements.
\end{definition}

In other words, $P_r$ is constructed as follows: (1) we pick $r$ columns of our matrix, (2) in each of these columns, we pick one of the two entries, (3) we make the sum of all these entries, (4) we make the product of all these sums, over all the possible choices.

Here are a few examples, for a matrix denoted $(^a_b)$:
\begin{eqnarray*}
P_0&=&0,\\
P_1&=&\prod_ia_ib_i,\\
P_2&=&\prod_{i\neq j}(a_i+a_j)(b_i+b_j)(a_i+b_j)(a_j+b_i),\\
P_3&=&\prod_{i\neq j\neq k\neq i}(a_i+a_j+a_k)(a_i+a_j+b_k)(a_i+b_j+b_k)(b_i+b_j+b_k).
\end{eqnarray*}

Observe that all these products are invariant under column permutation and flip.

\begin{proposition}
For any two-row matrix $(^a_b)\in M_{2\times q}(\mathbb N)$, the quantity
$$\Gamma\begin{pmatrix}a\\ b\end{pmatrix}
=\left[\frac{0}{AB}\cdot\frac{P_qP_{q-2}\ldots}{P_{q-1}P_{q-3}\ldots}\right],$$
where $A=\Sigma a_i,B=\Sigma b_i$, satisfies the balancing assumptions in Definition 2.8.
\end{proposition}

\begin{proof}
We have to show that for the fraction in the statement, the number of upper terms equals the number of lower terms, and the total upper sum equals the total lower sum. It is convenient to write our fraction in a closed form, as follows:
$$\Gamma\begin{pmatrix}a\\ b\end{pmatrix}=\left[\frac{0}{AB}
\prod_{r=0}^qP_r^{(-1)^{q-r}}\right].$$
Since $P_r$ is a product of $2^r(^q_r)$ terms, the signed total number of terms is:
$$\#=-1+\sum_{r=0}^q(-1)^{q-r}2^r\begin{pmatrix}q\\ r\end{pmatrix}=-1+(2-1)^q=0.$$
Let us look now at the total sums. These are all multiples of the quantity $A+B$, and for $r\geq 1$ the averaged total sum inside $P_r$ is given by:
$$\Sigma_r=\frac{1}{A+B}\sum_{S\in P_r}S=2^r\begin{pmatrix}q\\ r\end{pmatrix}\frac{r}{2q}=2^{r-1}\begin{pmatrix}q-1\\ r-1\end{pmatrix}.$$
Also, at $r=0$ we have $\Sigma_0=0$, so the averaged signed total sum is:
$$\Sigma=-1+\sum_{r=1}^q(-1)^{q-r}2^{r-1}\begin{pmatrix}q-1\\ r-1\end{pmatrix}=-1+(2-1)^{q-1}=0.$$
Thus $\Gamma$ is indeed balanced in the sense of Definition 2.8.
\end{proof}

\begin{definition}
We make the normalization
$$J\begin{pmatrix}a\\ b\end{pmatrix}
=\left[\frac{0}{AB}\cdot\frac{P_qP_{q-2}\ldots}{P_{q-1}P_{q-3}\ldots}\right]
\varphi\begin{pmatrix}a\\ b\end{pmatrix},$$
where $A=\Sigma a_i,B=\Sigma b_i$.
\end{definition}

As a first observation, at $q=1$ we have the good constant, namely:
$$\Gamma
=\left[\frac{0}{AB}\cdot\frac{P_1}{P_0}\right]
=\left[\frac{0}{ab}\cdot\frac{ab}{0}\right]=1.$$
At $q=2$ now, we recover the normalization constant in Proposition 4.1:
$$\Gamma
=\left[\frac{0}{AB}\cdot\frac{P_2P_0}{P_1}\right]
=\left[\frac{P_0^2(P_2/AB)}{P_1}\right]
=\left[\frac{0^2(a_1+b_2)(a_2+b_1)}{a_1a_2b_1b_2}\right].$$
Finally, at $q=3$ we recover the constant in Proposition 4.5:
$$\Gamma
=\left[\frac{0}{AB}\cdot\frac{P_3P_1}{P_2P_0}\right]
=\left[\frac{P_1(P_3/AB)}{P_2}\right]
=\left[\frac{\prod_ia_ib_i\prod_{i\neq j\neq k\neq i}(a_i+a_j+b_k)(a_i+b_j+b_k)}{\prod_{i\neq j}(a_i+a_j)(b_i+b_j)(a_i+b_j)(a_j+b_i)}\right].$$
In what follows we will prove that, in general, $\varphi$ is indeed a normalization in the sense of Definition 4.2, and that it satisfies all the extra properties listed in the introduction.

In order to prove these results, we use the following key lemma.

\begin{lemma}
We have the ``rational transmutation formula''
$$\Gamma_n\begin{pmatrix}a&c\\ b&0\end{pmatrix}
=\left[\frac{0(B+C)}{BC}\right]_n\Gamma_{n+C}\begin{pmatrix}a\\ b\end{pmatrix}$$
valid for any vectors $a,b\in\mathbb N^q$ and $c\in\mathbb N^s$.
\end{lemma}

\begin{proof}
We first prove the result at $s=1$. Let us introduce the following quantities:
$$P_r^+=P_r\begin{pmatrix}a&c\\ b&0\end{pmatrix},\quad
P_r=P_r\begin{pmatrix}a\\ b\end{pmatrix},\quad 
P_r^{(c)}=\prod_{S\in P_r(^a_b)}(S+c).$$
Let us look now at the quantity $P_r^+$. This is the product of sums obtained by selecting $r$ columns of the matrix $(^a_b{\ }^c_0)$, then selecting one entry in each of these columns, and making the sum. We can see that we have three cases: (1) either the $r$ selected columns all belong to the matrix $(^a_b)$, (2) or $r-1$ of these columns belong to $(^a_b)$, and the $0$ entry is selected from the last column, (3) or $r-1$ of these columns belong to $(^a_b)$, and the $c$ entry is selected from the last column. Thus we have the following formula:
$$P_r^+=P_rP_{r-1}P_{r-1}^{(c)}.$$
Here we make the convention that both quantities $P_{q+1},P_{-1}$ that can appear are the empty products. Now in terms of the $\Gamma$ quantity on the left, we get:
\begin{eqnarray*}
\Gamma_n\begin{pmatrix}a&c\\ b&0\end{pmatrix}
&=&\left[\frac{0}{(A+c)B}\cdot\frac{P_{q+1}^+P_{q-1}^+\ldots}{P_q^+P_{q-2}^+\ldots}\right]_n\\
&=&\left[\frac{0}{(A+c)B}\cdot\frac{(P_qP_{q-1}P_{q-2}\ldots)(P_q^{(c)}P_{q-2}^{(c)}\ldots)}{(P_qP_{q-1}P_{q-2}\ldots)(P_{q-1}^{(c)}P_{q-3}^{(c)}\ldots)}\right]_n\\
&=&\left[\frac{0}{(A+c)B}\cdot\frac{P_q^{(c)}P_{q-2}^{(c)}\ldots}{P_{q-1}^{(c)}P_{q-3}^{(c)}\ldots}\right]_n.
\end{eqnarray*}
As for the $\Gamma$ quantity on the right, this is given by:
\begin{eqnarray*}
\Gamma_{n+c}\begin{pmatrix}a\\ b\end{pmatrix}
&=&\left[\frac{0}{AB}\cdot\frac{P_qP_{q-2}\ldots}{P_{q-1}P_{q-3}\ldots}\right]_{n+c}\\
&=&\left[\frac{c}{(A+c)(B+c)}\cdot\frac{P_q^{(c)}P_{q-2}^{(c)}\ldots}{P_{q-1}^{(c)}P_{q-3}^{(c)}\ldots}\right]_n\\
&=&\left[\frac{Bc}{0(B+c)}\right]_n\cdot\left[\frac{0}{(A+c)B}\cdot\frac{P_q^{(c)}P_{q-2}^{(c)}\ldots}{P_{q-1}^{(c)}P_{q-3}^{(c)}\ldots}\right]_n.
\end{eqnarray*}
By comparing the above two formulae, this finishes the proof at $s=1$. For the general case now, $s\in\mathbb N$, we proceed by recurrence. By using the result at $s-1$ and at $1$, we get:
\begin{eqnarray*}
\Gamma_n\begin{pmatrix}a&c\\ b&0\end{pmatrix}
&=&\left[\frac{0(B+C-c_1)}{B(C-c_1)}\right]_n\Gamma_{n+C-c_1}\begin{pmatrix}a&c_1\\ b&0\end{pmatrix}\\
&=&\left[\frac{0(B+C-c_1)}{B(C-c_1)}\right]_n\cdot\left[\frac{0(B+c_1)}{Bc_1}\right]_{n+C-c_1}\Gamma_{n+C}\begin{pmatrix}a\\ b\end{pmatrix}\\
&=&\left[\frac{0(B+C-c_1)}{B(C-c_1)}\right]_n\cdot\left[\frac{(C-c_1)(B+C)}{(B+C-c_1)C}\right]_n\Gamma_{n+C}\begin{pmatrix}a\\ b\end{pmatrix}.
\end{eqnarray*}
By simplifying we obtain the formula in the statement, and this finishes the proof.
\end{proof}

\begin{proposition}
The function $\varphi$ constructed in Definition 5.3 is a normalization of $J$, in the sense of Definition 4.2.
\end{proposition}

\begin{proof}
We use Proposition 4.3. The $\Lambda$ function is given by:
$$\Lambda\begin{pmatrix}a\\ b\end{pmatrix}
=\left[\frac{AB}{0(A+B)}\right]\cdot\left[\frac{0}{AB}\cdot\frac{P_qP_{q-2}\ldots}{P_{q-1}P_{q-3}\ldots}\right]
=\left[\frac{1}{A+B}\cdot\frac{P_qP_{q-2}\ldots}{P_{q-1}P_{q-3}\ldots}\right].$$
Since all the products $P_r$, as well as $A+B$, are invariant under column permutation and flip, so is $\Lambda$. So, it remains to prove that $\Lambda$ has the good value on crosses.

For a cross $(^a_b{\ }^c_0)$ with $a,b\in 2\mathbb N$ and $c\in(2\mathbb N)^{q-1}$, the transmutation formula gives:
$$\Gamma_n\begin{pmatrix}a&c\\ b&0\end{pmatrix}
=\left[\frac{0(B+C)}{BC}\right]_n\Gamma_{n+C}\begin{pmatrix}a\\ b\end{pmatrix}
=\left[\frac{0(B+C)}{BC}\right]_n.$$
Thus $\Gamma$ has the good value on crosses, and so has $\Lambda$, and this finishes the proof.
\end{proof}

We are now in position of stating and proving one of our main results.

\begin{theorem}
The function $\varphi$ has the following properties:
\begin{enumerate}
\item Extension property: $\varphi(^a_b)=\varphi(^a_b{\ }^0_0)$.

\item Invariance under row permutation: $\varphi(^a_b)=\varphi(^b_a)$.

\item Invariance under column permutation: $\varphi(^a_b)=\varphi(^{\sigma(a)}_{\sigma(b)})$.

\item Invariance under flipping of columns: $\varphi(^a_b{\ }^c_d)=\varphi(^a_b{\ }^d_c)$.

\item Compression property: $\varphi(^a_b{\ }^c_0)=\varphi(^a_b{\ }^{\Sigma c_i}_{\ 0})$.

\item Transmutation: $\varphi_n(^a_b{\ }^c_0)=\varphi_{n+c}(^a_b)$ ($c\in\mathbb N$).

\item Triviality on crosses: $\varphi(^a_b{\ }^c_0)\in\{0,1\}$ ($a,b\in\mathbb N$).

\item Triviality at $n=2$: $\varphi_2(^a_b{\ }^c_d)\in\{-1,0,1\}$ ($a,b,c,d\in\mathbb N$).

\item Symmetry at $q=2$: $\varphi(^a_b{\ }^c_d)$ is symmetric ($a,b,c,d\in\mathbb N$).
\end{enumerate}
\end{theorem}

\begin{proof}
All the assertions follow from the results that we already have:

(1) This follows from the rational transmutation formula, at $c=0$.

(2,3,4) These assertions follow from Proposition 5.5.

(5) The rational transmutation formula shows that $\Gamma$ is invariant under compression. Since $J$ has as well this property, this gives the result.

(6) This follows by combining Proposition 3.5 and Lemma 5.4.

(7) This was already done, in Proposition 5.5.

(8,9) This was already done, in Proposition 4.1.
\end{proof}

\section{The three-row case}

In this section we present some advances on the $3\times 3$ case. Let us look first at the elementary expansion formula, at $p=3$. In this formula the combinatorics comes from the pairings of $a+b+c$ elements, and associated to such a pairing are the parameters $r,s,t$, describing the number of ``mixed'' $a-b$, $a-c$, $b-c$ pairings that can appear.

In terms of the vectors formed by the parameters $r,s,t$, the formula looks as follows.

\begin{proposition}
We have the ``three-row elementary expansion formula''
$$J\begin{pmatrix}a\\ b\\ c\end{pmatrix}
=\sum_{rst}K_{rst}'\begin{pmatrix}a\\ b\\ c\end{pmatrix}
J\begin{pmatrix}
1^R&1^S&0^T&A-R-S&0&0\\
1^R&0^S&1^T&0&B-R-T&0\\
0^R&1^S&1^T&0&0&C-S-T
\end{pmatrix},$$
where $A,B,C,R,S,T$ are the sums of entries of the vectors $a,b,c,r,s,t$.
\end{proposition}

\begin{proof}
This simply follows from Theorem 2.7. Observe that the constants appearing in the above formula are given by $K'_{rst}=\prod K_{r_js_jt_j}(a_j,b_j,c_j)/(a_j!!b_j!!c_j!!)$, where:
$$K_{rst}(a,b,c)=
\begin{pmatrix}a\\ r\end{pmatrix}
\begin{pmatrix}a\\ s\end{pmatrix}
\begin{pmatrix}b\\ r\end{pmatrix}
\begin{pmatrix}b\\ t\end{pmatrix}
\begin{pmatrix}c\\ s\end{pmatrix}
\begin{pmatrix}c\\ t\end{pmatrix}
r!s!t!(a-r-s)!!(b-r-t)!!(c-s-t)!!.$$
Indeed, this latter quantity counts the number of pairings of $a+b+c$ elements, having $r,s,t$ as the numbers of ``mixed'' $a-b$, $a-c$, $b-c$ pairings.
\end{proof}

The above formula shows, and in a very explicit way, that we are in a situation which is quite similar to the one in the two-row case, investigated in detail in \cite{int}: the compressed elementary matrices are simply the diagonal ones, ``decorated'' with $1$'s. 

So, in view of the results in \cite{int}, we are led to the following two questions.

\begin{problem}
What are the values of $J$ on the diagonal $3\times 3$ matrices? And, once these values are known, can linear algebra handle the ``decoration with $1$'' operation?
\end{problem}

In what follows we will focus on the first question. In order to deal with the diagonal  matrices, we use an ad-hoc linear algebra method. Consider the following functions:
$$J_c(x,y)=J\begin{pmatrix}a&&\\ &b&\\ x&y&c\end{pmatrix}.$$
Here, and in what follows, the matrix entries which are omitted are by definition 0. 

\begin{theorem}
The functions $J_c(x,y)$ satisfy the recurrence
$$J_{c+2}(x,y)=\frac{(L+n)J_c(x,y)-(x+1)J_c(x+2,y)-(y+1)J_c(x,y+2)}{c+n-2},$$
where $L=a+b+c+x+y$, with the following initial data:
$$J_0(x,y)=\left[\frac{0(a+b+x)(a+b+y)}{(a+b)(a+x)(b+y)}\right].$$
\end{theorem}

\begin{proof}
We use the same trick as in the proof of the cross formula. By inserting the trivial quantity $\Sigma u_{3i}^2=1$ on the third row, we obtain:
\begin{eqnarray*}
I\begin{pmatrix}a&&\\ &b&\\ x&y&c\end{pmatrix}
&=&I\begin{pmatrix}a&&\\ &b&\\ x+2&y&c\end{pmatrix}
+I\begin{pmatrix}a&&\\ &b&\\ x&y+2&c\end{pmatrix}\\
&&+I\begin{pmatrix}a&&\\ &b&\\ x&y&c+2\end{pmatrix}
+(n-3)I\begin{pmatrix}a&&\\ &b&\\ x&y&c&2\end{pmatrix}.
\end{eqnarray*}
We perform now the ``soft normalization'' in Definition 2.5. Most of the constants cancel, and we obtain the following formula, where $L=a+b+c+x+y$:
\begin{eqnarray*}
(L+n)J\begin{pmatrix}a&&\\ &b&\\ x&y&c\end{pmatrix}
&=&(x+1)J\begin{pmatrix}a&&\\ &b&\\ x+2&y&c\end{pmatrix}
+(y+1)J\begin{pmatrix}a&&\\ &b&\\ x&y+2&c\end{pmatrix}\\
&&+(c+1)J\begin{pmatrix}a&&\\ &b&\\ x&y&c+2\end{pmatrix}
+(n-3)J\begin{pmatrix}a&&\\ &b&\\ x&y&c&2\end{pmatrix}.
\end{eqnarray*}
The point now is that by doing a compression, we see that the last two $J$ quantities appearing on the right are equal. So, the above formula reads:
$$(L+n)J_c(x,y)=(x+1)J_c(x+2,y)+(y+1)J_c(x,y+2)+(c+n-2)J_{c+2}(x,y).$$
This gives the recurrence formula in the statement. Finally, the initial data is a value of $J$ on a transposed ``spark'', given by the formula in the statement. 
\end{proof}

Summarizing, we have a linear algebra method for computing altogether the quantities of type $J_c(x,y)$, and in particular the ``diagonal'' quantities of type $J_c(0,0)$.

As a first application of this method, let us record the following result.

\begin{proposition}
We have the following formula:
$$J\begin{pmatrix}a&&\\ &b&\\ &&2\end{pmatrix}
=\left[\frac{0(a+b)}{(a+2)(b+2)}\right]
\left(1+\frac{(a+b+n)(a+n-2)(b+n-2)}{n-2}\right).$$
\end{proposition}

\begin{proof}
We use Theorem 6.3, with $c=x=y=0$. We have $L=a+b$, and the quantity in the statement is therefore given by:
\begin{eqnarray*}
J_2(0,0)
&=&\frac{1}{n-2}\left((a+b+n)J_0(0,0)-J_0(2,0)-J_0(0,2)\right)\\
&=&\frac{1}{n-2}\left((a+b+n)\left[\frac{0(a+b)}{ab}\right]
-\left[\frac{0(a+b+2)}{(a+2)b}\right]
-\left[\frac{0(a+b+2)}{a(b+2)}\right]\right)\\
&=&\frac{1}{n-2}\left[\frac{0(a+b)}{ab}\right]
\left((a+b+n)-\frac{a+b+n-1}{a+n-1}-\frac{a+b+n-1}{b+n-1}\right).
\end{eqnarray*}
Thus we have a formula of the following type:
$$J_2(0,0)=\left[\frac{0(a+b)}{(a+2)(b+2)}\right]\cdot\frac{K}{n-2}.$$
More precisely, with $S=a+b+n$, $A=a+n-2$, $B=b+n-2$, we have:
\begin{eqnarray*}
K
&=&S(A+1)(B+1)-(S-1)(S+n-2)\\
&=&S(AB+A+B+1)-(S-1)(S+n-2)\\
&=&S(AB+S+n-3)-(S-1)(S+n-2)\\
&=&SAB+S(S+n-3)-(S-1)(S+n-2)\\
&=&SAB+(n-2).
\end{eqnarray*}
This gives the formula in the statement.
\end{proof}

The above computation is quite instructive, because it shows that we cannot have a flipping principle in the $3\times 3$ case. Indeed, by flipping the ``2'' entry of the matrix in Proposition 6.4 we obtain a cross, and the value of $J$ on this cross is known to be a product of factorials. On the other hand, the value of $J$ on the matrix in Proposition 6.4 is not a product of factorials. Thus, we cannot have a $3\times 3$ flipping principle.

\section{Moments of 3 coordinates}

The joint moments of $u_{11},u_{22}$ were explicitely computed in \cite{int}. In this section we present an explicit formula for the joint moments of $u_{11},u_{22},u_{33}$.

We begin our study by solving the recurrence in Theorem 6.3.

\begin{proposition}
We have the integration formula
\begin{eqnarray*}
J\begin{pmatrix}a&&\\ &b&\\ x&y&c\end{pmatrix}
&=&\sum_{k=0}^{c/2}(-1)^k\begin{pmatrix}c/2\\ k\end{pmatrix}
\sum_{r+s=k}
\begin{pmatrix}k\\ r\end{pmatrix}
\left[\frac{x+2r}{x}\right]_2\left[\frac{y+2s}{y}\right]_2\\
&&\left[\frac{0}{c}\right]_{n-1}
\cdot\left[\frac{L}{L-c+2k}\right]_{n+1}
\cdot\left[\frac{0(a+b+x+2r)(a+b+y+2s)}{(a+b)(a+x+2r)(b+y+2s)}\right]_n,
\end{eqnarray*}
where $L=a+b+c+x+y$.
\end{proposition}

\begin{proof}
According to Theorem 6.3, with $L=a+b+c+x+y$, we have:
$$J_{c+2}(x,y)=\frac{(L+n)J_c(x,y)-(x+1)J_c(x+2,y)-(y+1)J_c(x,y+2)}{c+n-2}.$$
By applying this formula $k$ times, we obtain:
\begin{eqnarray*}
J_{c+2}(x,y)
&=&\frac{(c-2k+n-1)!!}{(c+n-1)!!}\sum_{r+s+t=k}(-1)^{r+s}\frac{k!}{r!s!t!}\cdot\frac{(L+n+1)!!}{(L-2t+n+1)!!}\\
&&\frac{(x+2r)!!}{x!!}\cdot\frac{(y+2s)!!}{y!!}\cdot J_{c-2k+2}(x+2r,y+2s).
\end{eqnarray*}
With the change $c\to c-2$, this formula becomes:
\begin{eqnarray*}
J_c(x,y)
&=&\frac{(c-2k+n-3)!!}{(c+n-3)!!}\sum_{r+s+t=k}(-1)^{r+s}\frac{k!}{r!s!t!}\cdot\frac{(L+n-1)!!}{(L-2t+n-1)!!}\\
&&\frac{(x+2r)!!}{x!!}\cdot\frac{(y+2s)!!}{y!!}\cdot J_{c-2k}(x+2r,y+2s).
\end{eqnarray*}
With the choice $k=c/2$, this formula becomes:
\begin{eqnarray*}
J_c(x,y)
&=&\frac{(n-3)!!}{(c+n-3)!!}\sum_{r+s+t=c/2}(-1)^{r+s}\frac{(c/2)!}{r!s!t!}\cdot\frac{(L+n-1)!!}{(L-2t+n-1)!!}\\
&&\frac{(x+2r)!!}{x!!}\cdot\frac{(y+2s)!!}{y!!}\cdot J_0(x+2r,y+2s).
\end{eqnarray*}
By using the initial data in Theorem 6.3, we get:
\begin{eqnarray*}
J\begin{pmatrix}a&&\\ &b&\\ x&y&c\end{pmatrix}
&=&\frac{(n-3)!!}{(c+n-3)!!}\sum_{r+s+t=c/2}(-1)^{r+s}\frac{(c/2)!}{r!s!t!}\cdot\frac{(L+n-1)!!}{(L-2t+n-1)!!}\\
&&\frac{(x+2r)!!}{x!!}\cdot\frac{(y+2s)!!}{y!!}
\cdot\left[\frac{0(a+b+x+2r)(a+b+y+2s)}{(a+b)(a+x+2r)(b+y+2s)}\right].
\end{eqnarray*}
With $k=r+s$, this gives the formula in the statement.
\end{proof}

\begin{theorem}
The joint moments of $u_{11},u_{22},u_{33}$ are given by
\begin{eqnarray*}
J\begin{pmatrix}a&&\\ &b&\\ &&c\end{pmatrix}
&=&\sum_{k=0}^{c/2}(-1)^k\begin{pmatrix}c/2\\ k\end{pmatrix}
\frac{k!}{2^k}
\sum_{r+s=k}
\begin{pmatrix}2r\\ r\end{pmatrix}\begin{pmatrix}2s\\ s\end{pmatrix}\\
&&\left[\frac{0}{c}\right]_{n-1}
\cdot\left[\frac{a+b+c}{a+b+2k}\right]_{n+1}
\cdot\left[\frac{0(a+b+2r)(a+b+2s)}{(a+b)(a+2r)(b+2s)}\right]_n,
\end{eqnarray*}
where $J$ denotes as usual the normalization of $I$ made in Definition 2.4.
\end{theorem}

\begin{proof}
We use Proposition 7.1 at $x=y=0$. The second line of terms is the one in the statement, and the coefficient after the first sum is:
$$\begin{pmatrix}k\\ r\end{pmatrix}
\left[\frac{2r}{0}\right]_2\left[\frac{2s}{0}\right]_2
=\frac{k!}{r!s!}(2r)!!(2s)!!
=\frac{k!}{r!s!}\cdot\frac{(2r)!}{2^rr!}\cdot\frac{(2s)!}{2^ss!}
=\frac{k!}{2^k}\begin{pmatrix}2r\\ r\end{pmatrix}\begin{pmatrix}2s\\ s\end{pmatrix}.$$
Thus we get the formula in the statement.
\end{proof}

The above formula is of course not very beautiful, but can be quickly implemented on a computer. The experiments based on it lead for instance to the following conjecture.

\begin{conjecture}
For any $a\geq b\geq c$, the following quantity is an integer:
$$\left[\frac{c}{0}\right]_{n-1}\left[\frac{b+c}{0}\right]_nJ\begin{pmatrix}a&&\\ &b&\\ &&c\end{pmatrix}.$$
\end{conjecture}

We do not know for the moment how to prove this conjecture. Nor do we have an advance on the quite similar conjectures formulated in \cite{int}, in the $2\times 2$ case. We believe that a good framework for the study of these questions is that of the upper triangular $3\times 3$ matrices, but for the moment we do not have an explicit integration formula here.

\section{Hyperspherical variables}

In this section and in the next one we explore yet another point of view on the computation of the integrals $I(a)$. We will develop here some direct geometric techniques.

Our starting point is the Euler-Rodrigues formula. This formula parametrizes the rotations of $\mathbb R^3$, by points of the sphere $S^3\subset\mathbb R^4$, as follows:
$$u=\begin{pmatrix}
x^2+y^2-z^2-t^2&2(yz-xt)&2(xz+yt)\\
2(xt+yz)&x^2+z^2-y^2-t^2&2(zt-xy)\\
2(yt-xz)&2(xy+zt)&x^2+t^2-y^2-z^2
\end{pmatrix}.$$
Here $u\in SO_3$ is an arbitrary rotation of $\mathbb R^3$, and $(x,y,z,t)\in S^3$, which is uniquely determined up to a sign change, comes from the double cover $S^3=SU_2\to SO_3$.

A natural problem, to be explored in what follows, is whether we have similar models in higher dimensions. For instance one can ask for a ``probabilistic'' modelling of the standard matrix coordinates of the group $SO_{n+1}$ with $n\in\mathbb N$ arbitrary, by variables over the sphere $S^{2n-1}$. So, consider the space $\mathbb R^N$, with the standard coordinates denoted $x_1,\ldots,x_N$. When $N$ is small, or when we will not need all these coordinates, these will be simply denoted $x,y,z,\ldots$ We denote by $S^{N-1}\subset\mathbb R^N$ the sphere.

\begin{proposition}
The following variables have the same law:
\begin{enumerate}
\item $\Sigma x_i^2-\Sigma y_i^2:S^{2n-1}\to\mathbb R$,

\item $x:S^n\to\mathbb R$.
\end{enumerate}
\end{proposition}

\begin{proof}
We can write the metric of $S^{2n-1}$ in terms of the metrics of two copies of $S^{n-1}$ at distance $\pi/2$:
$$g_{2n-1}=\cos^2t\cdot g_{n-1}(x)+\sin^2t\cdot g_{n-1}(y)+dt^2.$$
Here $t$ is the ``position'' between the two copies of $S^{n-1}$, which is $0$ on one and $\pi/2$ on the other. This shows that the density of $t$ is given by:
$$c_n^2 \cos^{n-1}t\sin^{n-1}t\,dt=c\sin^{n-1}2t\,dt.$$
Here $c_n$ is the volume de $S^{n-1}$, and $c$ is another constant. Now our function is:
$$u=\cos^2t-\sin^2t=\cos 2t.$$
This shows that the law of $u$ is $c(1-u^2)^{n/2-1}\,du$, and this gives the result.
\end{proof}

The above result is quite interesting, because at $n=2$ we can obtain exactly, modulo some rotations, the 9 variables appearing in the Euler-Rodrigues formula. This will be explained in detail in the next section. For the moment let us present a ``complex variable'' extension of Proposition 8.1, to be used as well in the next section. We denote the coordinates of $\mathbb R^{2n}$ by $x_1,\ldots,x_n,y_1,\ldots,y_n$. We use as well the complex coordinates $z_i=x_i+iy_i$, coming from the identification $\mathbb R^{2n}=\mathbb C^n$.

\begin{theorem}
The following variables have the same law:
\begin{enumerate}
\item $\Sigma z_i^2:S^{2n-1}\to\mathbb C$,

\item $z:S^n\to\mathbb C$.
\end{enumerate}
\end{theorem}

\begin{proof}
Let $N=2n$, and $c_N$ be the volume of $S^{N-1}$. We will prove that for both variables in the statement, if we write them as $Z=X+iY$, then the law of $(X,Y)$ is given by: 
$$d\mu(X,Y)=\frac{c_{n-1}c_n}{2^{n-1}c_{2n}}(1-X^2-Y^2)^{(n-3)/2}\,dXdY.$$

(1) Let $S_1,S_2$ be the intersections of $S^{2n-1}$ with $\mathbb R^n,(i\mathbb R)^n$. Let $g_1,g_2$ be the metrics on these spheres, and let $\theta$ be the distance to $S_1$. The metric on $S^{2n-1}$ can be written as: 
$$d\theta^2+\cos^2 \theta\cdot g_1+\sin^2\theta\cdot g_2.$$
This shows that the law of $\theta$ is given by:
$$\mu(\theta)=\frac{c_n^2}{c_{2n}}\cos^{n-1}\theta\sin^{n-1}\theta\,d\theta
= \frac{c_n^2}{2^{n-1}c_{2n}}\sin^{n-1}2\theta\,d\theta.$$
We have $X=\cos^2\theta-\sin^2\theta=\cos 2\theta$, so the law of $X$ is:
$$\mu(X)=\frac{c_n^2}{2^{n-1}c_{2n}}\sin^{n-2}2\theta\,dX=\frac{c_n^2}{2^{n-1}c_{2n}}(1-X^2)^{(n-2)/2}dX.$$
We now consider $X$ fixed. For $z=(z_1,\ldots, z_n)\in S^{2n-1}$,  we set $u=Re(z)/||Re(z)||$ and $v=Im(z)/||Im(z)||$. Let $\alpha$ be the angle between $u$ and $v$. The law of $\alpha$ is:
$$\frac{c_{n-1}}{c_n}\sin^{n-2}\alpha\,d\alpha.$$
Now by definition of $Y$ we have:
$$Y=2\sin\theta\cos\theta\langle u,v\rangle=\sin 2\theta\cos\alpha.$$
So the law of $Y$, for $X$ fixed, is: 
\begin{eqnarray*}
\nu_X(Y) 
&=&\frac{c_{n-1}}{c_n}\cdot\frac{\sin^{n-3}\alpha}{\sin 2\theta}\,dY\\
&=&\frac{c_{n-1}}{c_n}\cdot\frac{\sin^{n-3}\alpha}{\sin 2\theta}\,dY\\
&=&\frac{c_{n-1}}{c_n}\cdot\frac{(1-Y^2/\sin^22\theta)^{(n-3)/2}}{\sin 2\theta}\,dY\\
&=&\frac{c_{n-1}}{c_n}\cdot\frac{(1-X^2-Y^2)^{(n-3)/2}}{(1-X^2)^{(n-3)/2}\sqrt{1-X^2}}\,dY\\
&=&\frac{c_{n-1}}{c_n}\cdot\frac{(1-X^2-Y^2)^{(n-3)/2}}{(1-X^2)^{(n-2)/2}}\,dY.
\end{eqnarray*}
Together with the law of $X$ found above, this gives the result.

(2) Let $\beta_1$ be the angle with $e_1=(1,0,\ldots,0)$, so that $x_1=\cos\beta_1$. The law of $\beta_1$ is:  
$$\mu(\beta_1)=\frac{c_n}{c_{n+1}}\sin^{n-1}\beta_1\,d\beta_1.$$
The law of $x_1$ is therefore: 
$$\mu(x_1)= \frac{c_n}{c_{n+1}}\sin^{n-2}\beta_1\,dx_1
=\frac{c_n}{c_{n+1}}(1-x_1^2)^{(n-2)/2}\,dx_1.$$
Now let $x_1$ be fixed, and let $\beta_2$ be the angle with $e_2=(0,1,0,\ldots,0)$. The law of $\beta_2$ is:
$$ \mu_{x_1}(\beta_2)=\frac{c_{n-1}}{c_n}\sin^{n-2}\beta_2\,d\beta_2.$$
We have $x_2=\sin\beta_1\cos\beta_2$, so the law of $x_2$  is:
\begin{eqnarray*}
d\mu_{x_1}(x_2) 
&=&\frac{c_{n-1}}{c_n}\sin^{n-3}\beta_2\frac{dx_2}{\sin\beta_1}\\
&=&\frac{c_{n-1}}{c_n}\left(1-\frac{x_2^2}{\sin^2\beta_1}\right)^{(n-3)/2}\frac{dx_2}{\sin\beta_1} \\
&=& \frac{c_{n-1}}{c_n}\cdot\frac{(1-x_1^2-x_2^2)^{(n-3)/2}}{(1-x_1^2)^{(n-2)/2}}\,dx_2.
\end{eqnarray*}
Together with the law of $x_1$ found above, this gives the result.
\end{proof}

\section{The modelling problem}

Let us go back to Proposition 8.1. This gives a model for the $S^n$-hyperspherical law, meaning the common law of the standard $n+1$ coordinates of $S^n\subset\mathbb R^{n+1}$, as a variable over the sphere $S^{2n-1}$. The point now is that we can obtain a whole family of variables over $S^{2n-1}$ which are $S^n$-hyperspherical, simply by using the action of $O_{2n}$.

We recall that the standard coordinates of $\mathbb R^{2n}$ are denoted $x_1,\ldots,x_n,y_1,\ldots,y_n$. So, the points of $\mathbb R^{2n}$ are the column vectors $(^x_y)$, and the action of an orthogonal matrix $U\in O_{2n}$ can be written as $U(^x_y)=(^{Ux}_{Uy})$, where $Ux,Uy$ are as usual column vectors.

\begin{proposition}
The following variables have the same law:
\begin{enumerate}
\item $\xi(U)=\Sigma Ux_i^2-\Sigma Uy_i^2:S^{2n-1}\to\mathbb R$, for any $U\in O_{2n}$,

\item $u_{ij}:SO_{n+1}\to\mathbb R$, for any $i,j\in\{1,\ldots,n+1\}$.
\end{enumerate}
\end{proposition}

\begin{proof}
This is clear from Proposition 8.1, and from the invariance of the uniform measure on the sphere $S^{2n-1}$ under the action of $O_{2n}$, i.e. from $d(^x_y)=dU(^x_y)$.
\end{proof}

The 3 diagonal variables in the Euler-Rodrigues formula are obviously all of the form $\xi(U)$, and the same holds in fact for the 6 off-diagonal variables. This key observation, to be discussed in detail a bit later, suggests the following general question.

\begin{problem}
Given a subset $K\subset\{1,\ldots,n+1\}^2$, is it possible to find matrices $U^{ij}\in O_{2n}$ with $(i,j)\in K$, such that the following families have the same joint law:
\begin{enumerate}
\item $\xi(U^{ij}):S^{2n-1}\to\mathbb R$, with $(i,j)\in K$,

\item $u_{ij}:SO_{n+1}\to\mathbb R$, with $(i,j)\in K$.
\end{enumerate}
\end{problem}

As a first remark, depending on the shape of $K$, we have several problems: ``one-row problem'', ``$2\times 2$ problem'', ``diagonal problem'', and so on. The $(n+1)\times (n+1)$ problem, which of course solves all these questions, will be refered to as the ``global problem''.

In what follows we will discuss some simple solutions to some of these problems. It is convenient to represent a ``solution'' by the $K$-shaped array $U^K=\{U^{ij}|(i,j)\in K\}$.

\begin{proposition}
The $1\times 1$ problem has the solution $U^K=(1)$. More generally, we have the solution $U^K=(U)$, for any matrix $U\in O_{2n}$.
\end{proposition}

\begin{proof}
The first assertion follows from Proposition 8.1, because the variable $\xi(1)$ is the one appearing there. As for the second assertion, this follows from Proposition 9.1.
\end{proof}

Let us discuss now the cases $n=1,2$, where the modelling problem can be solved by using the well-known structure results for the groups $O_2,O_3$. In order to deal with the case $n=1$, let us introduce the following matrices:
$$\rho=\frac{1}{\sqrt{2}}\begin{pmatrix}1&1\\ -1&1\end{pmatrix},\quad\quad
\rho^{-1}=\frac{1}{\sqrt{2}}\begin{pmatrix}1&-1\\ 1&1\end{pmatrix}.$$
These are the $45^\circ$ counterclockwise and clockwise rotations in the plane.

\begin{proposition}
The global $n=1$ problem has the following solution:
$$U^K=\begin{pmatrix}1&\rho\\ \rho^{-1}&1\end{pmatrix}.$$
More generally, the $16$-element subgroup of $O_2$ generated by $S_2$ and by $\rho$ produces via $U\to\xi(U)$ the $4$-variable correlated set $\{\pm\cos t,\pm\sin t\}$, repeated $4$ times.
\end{proposition}

\begin{proof}
Let us denote by $x,y$ the coordinates of $\mathbb R^2$. The actions of the three matrices in the statement are as follows:
$$1\begin{pmatrix}x\\ y\end{pmatrix}=\begin{pmatrix}x\\ y\end{pmatrix},
\quad\quad
\rho\begin{pmatrix}x\\ y\end{pmatrix}=\frac{1}{\sqrt{2}}\begin{pmatrix}x+y\\ -x+y\end{pmatrix},
\quad\quad
\rho^{-1}\begin{pmatrix}x\\ y\end{pmatrix}=\frac{1}{\sqrt{2}}\begin{pmatrix}x-y\\ x+y\end{pmatrix}.$$
Thus the corresponding $\xi$ variables are as follows:
\begin{eqnarray*}
\xi(1)&=&x^2-y^2,\\
\xi(\rho)&=&1/2((x+y)^2-(x-y)^2)=2xy,\\
\xi(\rho^{-1})&=&1/2((x-y)^2-(x+y)^2)=-2xy.
\end{eqnarray*}
So, the solution proposed in the statement is as follows:
$$\xi(U^K)=\begin{pmatrix}x^2-y^2&2xy\\ -2xy&x^2-y^2\end{pmatrix}.$$
Consider on the other hand the standard parametrization of $SO_2$, namely:
$$u=\begin{pmatrix}\cos t&\sin t\\ -\sin t&\cos t\end{pmatrix}.$$
We already know fom Proposition 9.3 that all 8 entries of $\xi(U^K)$ and of $u$ follow a common law (the arcsine one). So, we just have to check that the correlation between $x^2-y^2,2xy$ is the same as the one between $\cos t,\sin t$, and this is clear.

Regarding now the last assertion, the first remark is that the group generated by $\rho$ and by the coordinate switch $\sigma=(^0_1{\ }^1_0)$ is the dihedral group $D_8$. This group is given by $D_8=\mathbb Z_8\rtimes\mathbb Z_2$, and its elements are $\{\rho^k,\sigma\rho^k|k=0,1,\ldots,7\}$. Now by symmetry reasons, these 16 elements produce via $U\to\xi(U)$ the same variables as the 4 elements $\{1,\sigma,\rho,\rho^{-1}\}$, repeated 4 times.  But these latter 4 elements produce the variables $\{\pm (x^2-y^2),\pm 2xy\}$, which are the same as the variables $\{\pm\cos t,\pm\sin t\}$, and this finishes the proof.
\end{proof}

Let us discuss now the case $n=2$. We will see that the global modelling problem here is more or less equivalent, probabilistically speaking, to the Euler-Rodrigues formula. Consider the $45^\circ$ counterclockwise and clockwise rotations in the $z_1,z_2$ planes: 
$$\rho=\frac{1}{\sqrt{2}}\begin{pmatrix}1&1&0&0\\ -1&1&0&0\\ 0&0&1&1\\ 0&0&-1&1\end{pmatrix},\quad\quad
\rho^{-1}=\frac{1}{\sqrt{2}}\begin{pmatrix}1&-1&0&0\\ 1&1&0&0\\ 0&0&1&-1\\ 0&0&1&1\end{pmatrix}.$$
Let us introduce as well the following matrices:
$$\tau=\begin{pmatrix}
1&0&0&0\\
0&0&1&0\\
0&0&0&1\\
0&1&0&0
\end{pmatrix},
\quad\quad
\tau^{-1}=\begin{pmatrix}
1&0&0&0\\
0&0&0&1\\
0&1&0&0\\
0&0&1&0
\end{pmatrix}.$$
Observe that these are the two cycles in $S_3$, acting on the last 3 coordinates.

\begin{proposition}
The global $n=2$ problem has a solution of the following type, where $\rho_{ij}$ are various compositions of $\rho,\rho^{-1}$ with the permutations in $S_4$:
$$U^K=\begin{pmatrix}1&\rho_{12}&\rho_{13}\\ \rho_{21}&\tau&\rho_{23}\\ \rho_{31}&\rho_{32}&\tau^{-1}\end{pmatrix}.$$
More generally, the subgroup of $O_4$ generated by $S_4$ and by $\rho$ produces via $U\to\xi(U)$ a number of copies of the variables $\pm u_{ij}$, with $(i,j)\in\{1,2,3\}^2$, where $u\in SO_3$.
\end{proposition}

\begin{proof}
According to the definition of $\tau,\tau^{-1}$, we have the following formulae:
\begin{eqnarray*}
\xi(1)&=&x^2+y^2-z^2-t^2,\\
\xi(\tau)&=&x^2+z^2-t^2-y^2,\\
\xi(\tau^{-1})&=&x^2+t^2-y^2-z^2.
\end{eqnarray*}
The upper matrices in the statement can be taken as follows:
$$\rho_{12}=\frac{1}{\sqrt{2}}\begin{pmatrix}
0&1&1&0\\
1&0&0&-1\\
0&1&-1&0\\
1&0&0&1
\end{pmatrix},
\rho_{13}=\frac{1}{\sqrt{2}}\begin{pmatrix}
1&0&1&0\\
0&1&0&1\\
1&0&-1&0\\
0&1&0&-1
\end{pmatrix},
\rho_{23}=\frac{1}{\sqrt{2}}\begin{pmatrix}
0&0&1&1\\
1&-1&0&0\\
0&0&1&-1\\
1&1&0&0
\end{pmatrix}.
$$
The action of these matrices is as follows:
$$\rho_{12}\begin{pmatrix}x\\ y\\ z\\ t\end{pmatrix}=\frac{1}{\sqrt{2}}\begin{pmatrix}y+z\\ x-t\\ y-z\\ x+t\end{pmatrix}
,\quad\quad
\rho_{13}\begin{pmatrix}x\\ y\\ z\\ t\end{pmatrix}=\frac{1}{\sqrt{2}}\begin{pmatrix}x+z\\ y+t\\ x-z\\ y-t\end{pmatrix}
,\quad\quad
\rho_{23}\begin{pmatrix}x\\ y\\ z\\ t\end{pmatrix}=\frac{1}{\sqrt{2}}\begin{pmatrix}z+t\\ x-y\\ z-t\\ x+y\end{pmatrix}.$$
Thus the corresponding $\xi$ variables are as follows:
\begin{eqnarray*}
\xi(\rho_{12})&=&1/2((y+z)^2+(x-t)^2-(y-z)^2-(x+t)^2)=2(yz-xt),\\
\xi(\rho_{13})&=&1/2((x+z)^2+(y+t)^2-(x-z)^2-(y-t)^2)=2(xz+yt),\\
\xi(\rho_{23})&=&1/2((z+t)^2+(x-y)^2-(z-t)^2-(x+y)^2)=2(zt-xy).
\end{eqnarray*}
As for the lower matrices in the statement, these can be taken as follows:
$$\rho_{21}=\frac{1}{\sqrt{2}}\begin{pmatrix}
0&1&1&0\\
1&0&0&1\\
0&1&-1&0\\
1&0&0&-1
\end{pmatrix},
\rho_{31}=\frac{1}{\sqrt{2}}\begin{pmatrix}
1&0&-1&0\\
0&1&0&1\\
1&0&1&0\\
0&1&0&-1
\end{pmatrix},
\rho_{32}=\frac{1}{\sqrt{2}}\begin{pmatrix}
0&0&1&1\\
1&1&0&0\\
0&0&1&-1\\
1&-1&0&0
\end{pmatrix}.
$$
The action of these matrices is as follows:
$$\rho_{21}\begin{pmatrix}x\\ y\\ z\\ t\end{pmatrix}=\frac{1}{\sqrt{2}}\begin{pmatrix}y+z\\ x+t\\ y-z\\ x-t\end{pmatrix}
,\quad
\rho_{31}\begin{pmatrix}x\\ y\\ z\\ t\end{pmatrix}=\frac{1}{\sqrt{2}}\begin{pmatrix}x-z\\ y+t\\ x+z\\ y-t\end{pmatrix}
,\quad
\rho_{32}\begin{pmatrix}x\\ y\\ z\\ t\end{pmatrix}=\frac{1}{\sqrt{2}}\begin{pmatrix}z+t\\ x+y\\ z-t\\ x-y\end{pmatrix}.$$
Thus the corresponding $\xi$ variables are as follows:
\begin{eqnarray*}
\xi(\rho_{21})&=&1/2((y+z)^2+(x+t)^2-(y-z)^2-(x-t)^2)=2(xt+yz),\\
\xi(\rho_{31})&=&1/2((x-z)^2+(y+t)^2-(x+z)^2-(y-t)^2)=2(yt-xz),\\
\xi(\rho_{32})&=&1/2((z+t)^2+(x+y)^2-(z-t)^2-(x-y)^2)=2(xy+zt).
\end{eqnarray*}
Summing up, the solution proposed in the statement is as follows:
$$\xi(U^K)=\begin{pmatrix}
x^2+y^2-z^2-t^2&2(yz-xt)&2(xz+yt)\\
2(xt+yz)&x^2-y^2+z^2-t^2&2(zt-xy)\\
2(yt-xz)&2(xy+zt)&x^2-y^2-z^2+t^2
\end{pmatrix}.$$
But this is exactly the Euler-Rodrigues formula, and this finishes the proof.
\end{proof}

\begin{theorem}
The $1\times 2$ modelling problem has the solution $U^K=\begin{pmatrix}1& \rho\end{pmatrix}$, where $\rho$ is the $45^\circ$ rotation in each $z_i$ plane.
\end{theorem}

\begin{proof}
This is just a reformulation of Theorem 8.2, by taking real and imaginary parts, and by using the various notations and identifications from the beginning of this section.
\end{proof}

\section{Concluding remarks}

We have seen in this paper that the computation of the integrals $I(a)$ leads to a number of concrete questions regarding the symmetries and normalization of such integrals, as well as to a number of questions regarding their probabilitic modelling, over spheres.

All these questions depend in a crucial way on the shape $p\times q$ of our matrix $a$.

Most of the open problems that we have at this time concern the case of $3\times 3$ matrices. So, the main problem that we would like to raise here is that of understanding the integrals $I(a)$ for matrices of type $a\in M_{3\times 3}(\mathbb N)$, perhaps taken to be upper triangular.

In addition, we have the question of finding noncommutative analogues of these results and techniques. The subject here is quite interesting as well, because it is closely related to the meander determinant problematics, considered in \cite{dif}, \cite{dgg}. See \cite{gra}.

We intend to come back to these questions in some future work.

\end{document}